\documentclass[12pt]{amsart}
\usepackage[utf8]{inputenc}
\usepackage{mathrsfs} 
\usepackage{amscd}
\usepackage[all,cmtip]{xy}
\usepackage{amssymb}
\usepackage{bm}
\usepackage{graphicx}
\usepackage[centertags]{amsmath}
\usepackage{amsfonts}
\usepackage{amsthm}
\usepackage{enumerate}
\usepackage{tocvsec2}
\usepackage{xcolor}
\usepackage[margin=1in]{geometry}

\newtheorem{theorem}{Theorem}[section]
\newtheorem*{theorem*}{Theorem}

\newtheorem{corollary}[theorem]{Corollary}
\newtheorem{lemma}[theorem]{Lemma}
\newtheorem{rem}[theorem]{Remark}

\newtheorem{proposition}[theorem]{Proposition}

\newtheorem{fact}[theorem]{Fact}

\theoremstyle{definition}

\newcommand{\ee}{\varepsilon}
\newcommand{\nn}{\mathbb{N}}
\newcommand{\rr}{\mathbb{R}}

\begin{document}
\title{Concerning $q$-summable Szlenk index}

\begin{abstract} For each ordinal $\xi$ and each $1\leqslant q<\infty$, we define the notion of $\xi$-$q$-summable Szlenk index. When $\xi=0$ and $q=1$, this recovers the usual notion of summable Szlenk index. We define for an arbitrary weak$^*$-compact set a transfinite, asymptotic analogue $\alpha_{\xi,p}$ of the martingale type norm of an operator. We prove that this quantity is determined by norming sets and determines $\xi$-Szlenk power type and $\xi$-$q$-summability of Szlenk index. This fact allows us to prove that the behavior of operators under the $\alpha_{\xi,p}$ seminorms passes in the strongest way to injective tensor products of Banach spaces.  Furthermore, we combine this fact with a result of Schlumprecht to prove that a separable Banach space with good behavior with respect to the $\alpha_{\xi,p}$ seminorm can be embedded into a Banach space with a shrinking basis and the same behavior under $\alpha_{\xi,p}$, and in particular it can be embedded into a Banach space with a shrinking basis and the same $\xi$-Szlenk power type. Finally, we completely elucidate the behavior of the $\alpha_{\xi,p}$ seminorms under $\ell_r$ direct sums. This allows us to give an alternative proof of a result of Brooker regarding Szlenk indices of $\ell_p$ and $c_0$ direct sums of operators.

\end{abstract}

\author{R.M. Causey}
\address{Department of Mathematics, Miami University, Oxford, OH 45056, USA}
\email{causeyrm@miamioh.edu}

\thanks{2010 \textit{Mathematics Subject Classification}. Primary: 46B03, 46B06; Secondary: 46B28, 47B10.}
\thanks{\textit{Key words}: Szlenk index, operator ideals, ordinal ranks.}

\maketitle

\tableofcontents

\addtocontents{toc}{\setcounter{tocdepth}{1}}

\section{Introduction}

Two of the most important renorming theorems in Banach space theory is Enflo's result \cite{E} that a superreflexive Banach space can be renormed to be uniformly convex, and Pisier's result \cite{P} that any superreflexive Banach space can be renormed to be uniformly convex (resp. uniformly smooth) with a power type modulus of uniform convexity (resp. uniform smoothness). Pisier's proof also gave an exact characterization of which power types were possible in terms of two isomorphic invariants of the space involving Walsh-Paley martingale difference sequences and their domination of (resp. domination by) the canonical $\ell_p$ basis. Beauzamy \cite{Beau} proved an operator version of Enflo's result: An operator is super weakly compact if and only if it can be renormed to be uniformly convex. However, there is no operator version of Pisier's result, since the automatic power type obtained by Pisier for Banach spaces has to do with the submultiplicative nature of Haar type/cotype ideal norms for Banach spaces, which are not submultiplicative for operators. The same submultiplicative behavior of ideal norms and seminorms for spaces which are not submultiplicative for operators is by now a well-observed phenomenon, occurring with Rademacher, gaussian, Haar, martingale, and the recently defined asymptotic basic and asymptotic type and cotype.  Since there are operators which have a given property but which do not have automatic power type (for example, operators which are super weakly compact without having non-trivial Haar type or cotype), it has become standard to investigate the notions of (Haar, Rademacher, etc.) subtype and subcotype of operators, defined by whether or not the operator exhibits the worst possible behavior with respect to a sequence of norms or seminorms. This practice has been undertaken by Beauzamy  \cite{Beau} for Radmacher subtype and subcotype to characterize when $\ell_1$ or $c_0$ is crudely finitely representable in an operator; by Hinrichs \cite{Hi} for gaussian subcotype to characterize when $c_0$ is crudely finitely representable in an operator; by Wenzel \cite{We} for martingale and Haar subtype and subcotype to characterize super weak compactness; and Draga, Kochanek, and the author  \cite{CDK} to characterize when an operator is asymptotically uniformly smoothable, when $\ell_1$ is asymptotically crudely finitely representable in an operator, and when $c_0$ is asymptotically crudely finitely representable in an operator. 

The notion of asymptotic uniform smoothability is also of significant interest, and is fundamentally connected to the Szlenk index of a Banach space. A Banach space admits an equivalent asymptotically uniformly smooth norm if and only if its Szlenk index does not exceed $\omega$. This was shown by Knaust, Odell, and Schlulmprecht \cite{KOS} for the separable case and Raja \cite{R} for the general case. Lancien, Prochazka, and Raja \cite{LPR} proved a result analogous to that of Knaust, Odell and Schlumprecht for separable Banach spaces having Szlenk index not exceeding $\omega^{\xi+1}$ for every countable ordinal $\xi$, and a non-separable, operator version was proved in \cite{CD} for every (not necessarily countable) ordinal $\xi$.   Furthermore, Godefroy, Kalton, and Lancien \cite{GKL} gave a precise renorming theorem for a separable Banach space in terms of the Szlenk power type of the Banach space. This was generalized to non-separable spaces and operators in \cite{C3}, as well as to higher ordinals in terms of the behavior of special convex combinations of the branches of $n$-leveled weakly null trees where each level has specified order. Further renorming results were established in \cite{AUF}, analogous to those of Pisier in \cite{P} were investigated in terms of special convex combinations of the branches of $\omega$-leveled weakly null trees where each level has specified order.  The downside of the renorming results from \cite{AUF} is that they are only produce non-trivial equivalent norms when an operator or a space has some power type behavior of the $\ee$-Szlenk indices. That is, for a fixed $\xi$, the results of \cite{AUF} produce non-trivial equivalent norms for a Banach space $X$ or an operator $A:X\to Y$ if $\omega^\xi k_\ee \leqslant Sz(X, \ee)\leqslant \omega^\xi (k_\ee+1)$ (resp. $\omega^\xi k_\ee\leqslant Sz(A, \ee)\leqslant \omega^\xi (k_\ee+1)$) where $k_\ee\in \nn$ satisfies $\lim_{\ee\to 0^+}k_\ee=\infty$ but for some $1\leqslant q<\infty$, $\sup_{\ee>0} \ee^q k_\ee<\infty$. As was explained in \cite{C3}, for $\xi=0$, there exists an operator $A$ with Szlenk index $\omega$  which does not have this power type behavior, and for any $\xi>0$, there is a Banach space $X$ with Szlenk index $\omega^{\xi+1}$ which does not have this power type behavior.   This incompleteness in the renorming theorems from \cite{AUF} is the primary motivation to define the notion of subtype and investigate the behavior of operators with respect to a sequence of seminorms. 

For each ordinal $\xi$, each $1\leqslant p\leqslant \infty$, each $n\in \nn$, and each operator $A:X\to Y$, we define the quantity $\alpha_{\xi, p,n}(A)$.  The assignments $\alpha_{\xi, p, n}$ will act as our sequence of ideal seminorms. We will be interested in when an operator satisfies $\sup_n \alpha_{\xi, p,n}(A)<\infty$, which is trivial in the case $p=1$. For $1<p\leqslant \infty$, the worst possible behavior would be for the operator $A$ to satisfy $\underset{n}{\lim\sup} \alpha_{\xi,p,n}(A)/n^{1-1/p}>0$, so our notion of subtype will be defined by $\lim_n \alpha_{\xi,p,n}(A)/n^{1-1/p}=0$.  For each ordinal $\xi$ and $1<p\leqslant \infty$, we let $\mathfrak{A}_{\xi,p}$ denote the class of operators for which $\sup_n \alpha_{\xi,p,n}(A)<\infty$. 

For each ordinal $\xi$ and each $1\leqslant q<\infty$, we define what it means for an operator $A:X\to Y$ weak$^*$-compact subset $K$ of a dual Banach space $X^*$ to have $\xi$-$q$-\emph{summable Szlenk index}, which generalizes to other values of $q$ and higher ordinals the important notion of summable Szlenk index defined and studied in \cite{GKL}.   In what follows, we let $\mathfrak{D}_\xi$ denote the class of operators with Szlenk index not exceeding $\omega^\xi$ and for $1<p\leqslant \infty$, $\mathfrak{T}_{\xi,p}$ denotes the class of $\xi$-$p$-asymptotically uniformly smoothable operators.

\begin{theorem} For every ordinal $\xi$ and $1<p\leqslant \infty$, there exists an ideal norm $\mathfrak{a}_{\xi,p}$ on the class $\mathfrak{A}_{\xi,p}$ making $(\mathfrak{A}_{\xi,p}, \mathfrak{a}_{\xi,p})$ a Banach ideal. Moreover, if $1/p+1/q=1$, the class $\mathfrak{A}_{\xi,p}$ coincides with the class of operators having $\xi$-$q$-summable Szlenk index.    

An operator $A$ has Szlenk index not exceeding $\omega^{\xi+1}$ if and only if for some (equivalently, every) $1<p\leqslant \infty$, $\alpha_{\xi,p,n}(A)/n^{1-1/p}=0$.

 Furthermore, for any $1<r<p\leqslant \infty$, $$\mathfrak{D}_\xi \subsetneq \mathfrak{A}_{\xi,r} \subsetneq \mathfrak{T}_{\xi,p}\subset \mathfrak{A}_{\xi,p} \subsetneq \mathfrak{D}_{\xi+1}.$$

\end{theorem}

The quantity  $\alpha_{\xi,p}(\cdot)$ will be defined not only for operators, but for arbitrary, weak$^*$-compact subsets of the dual of a Banach space.  We prove the following regarding the quantities  $\alpha_{\xi,p}$  on weak$^*$-compact subsets of the dual of a Banach space. 

\begin{theorem}Let $X$ be a Banach space, $K\subset X^*$ weak$^*$-compact, $\xi$ an ordinal, and $1<p\leqslant \infty$.  Then $\alpha_{\xi,p}(K)=\alpha_{\xi,p}(\overline{\text{\emph{abs\ co}}}^{\text{weak}^*}(K))$. 

\label{arnold}
\end{theorem}

Combining Theorem \ref{arnold} with a result of Schlumprecht, we prove the following. 

\begin{theorem} Suppose $X$ is a  Banach space having a separable dual. Then there exists a Banach space $W$ with a shrinking basis such that for any countable ordinal and $1\leqslant p\leqslant \infty$, $\textbf{\emph{p}}_\xi(W)=\textbf{\emph{p}}_\xi(X)$ and if $\alpha_{\xi, p}(X)<\infty$, $\alpha_{\xi, p}(W)<\infty$.    

\label{schlumprecht}
\end{theorem}

Given operators $A_0:X_0\to X_1$, $A_1:X_1\to Y_1$, there is an induced operator $A_0\otimes A_1:X_0\hat{\otimes}_\ee X_1\to Y_0\hat{\otimes}_\ee Y_1$ between the injective tensor products.   By the geometric Hahn-Banach theorem, $(A_0\otimes A_1)^*(B_{(Y_0\hat{\otimes }_\ee Y_1)^*})$ is the weak$^*$-closed, convex hull of $\{x^*_0\otimes x^*_1: x^*_0\in A^*_0B_{Y^*_0}, A^*_1B_{Y^*_1}\}$. This fact combined with Theorem \ref{arnold} and a combinatorial lemma yields the following.

\begin{theorem} Fix any $1\leqslant q<\infty$, any ordinal $\xi$,  any operators $A_i:X_i\to Y_i$, $i=0,1$. Then if $A_0$, $A_1$ have $\xi$-$q$-summable Szlenk index, so does $A_0\otimes A_1$, and $\textbf{\emph{p}}_\xi(A_0\otimes A_1) \leqslant \max\{\textbf{\emph{p}}_\xi(A_0), \textbf{\emph{p}}_\xi(A_1)\}$.    

In the non-trivial case in which  $A_0, A_1\neq 0$,  $A_0\otimes A_1$ has $\xi$-$q$-summable Szlenk index if and only if $A_0$, $A_1$ do and $\textbf{\emph{p}}_\xi(A_0\otimes A_1) = \max\{\textbf{\emph{p}}_\xi(A_0), \textbf{\emph{p}}_\xi(A_1)\}$.

\end{theorem}

We also give a complete description of the behavior of the $\alpha_{\xi, r}$ seminorms under finite $\ell_p$ direct sums of operators. 

\begin{theorem} Fix $1\leqslant p\leqslant \infty$. Then for any finite set $\Lambda$, any ordinal $\xi$, and any operators $A_\lambda:X_\lambda\to Y_\lambda$, $$\alpha_{\xi,r}\bigl(\oplus_{\lambda\in \Lambda}A_\lambda:(\oplus_{\lambda\in \Lambda}X_\lambda)_{\ell_p(\Lambda)}\to (\oplus_{\lambda\in \Lambda}Y_\lambda)_{\ell_p(\Lambda)}\bigr)= \max_{\lambda\in \Lambda} \alpha_{\xi,r}(A_\lambda)$$ if $1\leqslant r\leqslant p$ and $$\alpha_{\xi,r}\bigl(\oplus_{\lambda\in \Lambda}A_\lambda:(\oplus_{\lambda\in \Lambda}X_\lambda)_{\ell_p(\Lambda)}\to (\oplus_{\lambda\in \Lambda}Y_\lambda)_{\ell_p(\Lambda)}\bigr)^s = \sum_{\lambda\in \Lambda} \alpha_{\xi,r}(A_\lambda)^s$$ if $p<r\leqslant \infty$, where $1/r+1/s=1/p$.

\end{theorem}

We also investigate the behavior of the $\alpha_{\xi,r}$ seminorms under infinite $\ell_p$ and $c_0$ directs sums of operators. In what follows, suppose that $\Lambda$ is a non-empty set and $\{A_\lambda:X_\lambda\to Y_\lambda:\lambda\in \Lambda\}$ is a uniformly bounded collection of Asplund operators.  We offer a new proof of a result due to Brooker concerning the Szlenk index of diagonal operators between $\ell_p$ direct sums. The method of our proof is dual to Brooker's method of proof, working with upper estimates on convex combinations of weakly null trees in the spaces rather than derivations in the dual. This result has many cases, so we isolate the following single case, which we believe to be of the greatest interest.

\begin{theorem} Fix  $p\in \{0\}\cup (1, \infty)$ and suppose that at least one of the operators $(A_\lambda)_{\lambda\in \Lambda}$ is non-compact. Then if $D^p:(\oplus_{\lambda\in \Lambda} X_\lambda)_{\ell_p(\Lambda)}\to (\oplus_{\lambda\in \Lambda} Y_\lambda)_{\ell_p(\Lambda)}$ is the diagonal operator such that $D^p|_{X_\lambda}=A_\lambda$,  $Sz(D^p)=\sup_{\lambda\in \Lambda} Sz(A_\lambda)$ if and only if for every $\ee>0$, $\sup_{\lambda\in \Lambda}Sz(A_\lambda, \ee)<\sup_{\lambda\in \Lambda} Sz(A_\lambda)$, and otherwise $Sz(D^p)= \bigl(\sup_{\lambda\in \Lambda} Sz(A_\lambda)\bigr)\omega$. 

Here, if $p=0$, we replace the $\ell_p(\Lambda)$ direct sum with the $c_0(\Lambda)$ direct sum.

\label{direct}
\end{theorem}

We also  elucidate the values of $\alpha_{\xi, \gamma}(D^p)$  (and therefore compute $\textbf{p}_\xi(D^p)$) for all values of $p\in \{0\}\cup [1, \infty]$ and $1\leqslant \gamma\leqslant 1$ only in terms of the behaviors of $Sz(A_\lambda, \ee)$ and $\alpha_{\xi, \gamma}(A_\lambda)$, thus completely describing the behavior of the ideal norms and Szlenk power type of diagonal operators between $\ell_p$ direct sums in terms of the ideal norms and Szlenk power types of the summands.

\begin{theorem} Fix $p\in \{0\}\cup (1, \infty)$ and suppose $Sz(D_p)=\omega^{\xi+1}$ and $\sup_{\lambda\in \Lambda} Sz(A_\lambda)=\omega^\zeta>1$.  Then by Theorem \ref{direct}, $\zeta+1=\xi+1$ if and only if there exists $\ee>0$ such that $\sup_{\lambda\in \Lambda} Sz(A_\lambda, \ee)=\omega^\zeta$, and otherwise $\zeta=\xi+1$. 

If $\zeta+1=\xi+1$, then for any $1\leqslant r\leqslant \infty$, then $\alpha_{\xi,r}(D_p)<\infty$ if and only if either $p=0$ or $1\leqslant r\leqslant p<\infty$.

If $\zeta=\xi+1$, then for any $1\leqslant r\leqslant \infty$, then $\alpha_{\xi,r}(D_p)<\infty$ if and only of one of the three following conditions hold: \begin{enumerate}[(i)]\item $p=0$ and $\sup_{\lambda\in \Lambda} \alpha_{\xi,r}(A_\lambda)<\infty$. \item $1\leqslant r\leqslant p<\infty$ and $\sup_{\lambda\in \Lambda} \alpha_{\xi,r}(A_\lambda)<\infty$.  \item $1<p<r\leqslant \infty$, $\sum_{\lambda\in \Lambda} \alpha_{\xi,r}(A_\lambda)^s<\infty$, and for any $\ee>0$, there exists a finite subset $\Upsilon$ of $\Lambda$ such that $\sup_{\lambda\in \Lambda\setminus \Upsilon} Sz(A_\lambda, \ee)<\omega^\xi$.   Here, $s$ is defined by  $1/r+1/s=1/p$. \end{enumerate}

\end{theorem}

\section{Combinatorial necessities}

\subsection{Trees of peculiar importance}

We first define some trees which will be of significant importance for us. Given a sequence $(\zeta_i)_{i=1}^n$ of ordinals and an ordinal $\zeta$, we let $\zeta+(\zeta_i)_{i=1}^n=(\zeta+\zeta_i)_{i=1}^n$. Given a set $G$ of sequences of ordinals and an ordinal $\zeta$, we let $\zeta+G=\{\zeta+t: t\in G\}$. For each $\xi\in \textbf{Ord}$ and $n\in \nn$, we define a tree $\Gamma_{\xi,n}$ which consists of decreasing ordinals in the interval $[0, \omega^\xi n)$. We let $$\Gamma_{0, 1}=\{(0)\}.$$  If $\xi$ is a limit ordinal and $\Gamma_{\zeta,1}$ has been defined for every $\zeta<\xi$, we let $$\Gamma_{\xi,1}=\bigcup_{\zeta<\xi}(\omega^\zeta+\Gamma_{\zeta+1,1}).$$  If for some $\xi$ and every $n\in \nn$, $\Gamma_{\xi,n}$ has been defined such that the first member of each sequence in $\Gamma_{\xi,n}$ lies in the interval $[\omega^\xi(n-1), \omega^\xi n)$, we let $$\Gamma_{\xi+1, 1}=\bigcup_{n=1}^\infty \Gamma_{\xi, n}.$$  Finally, if $\Gamma_{\xi, 1}$ has been defined, we let $\Lambda_{\xi,1,1}=\Gamma_{\xi, 1}$ and for $1<n\in \nn$ and $1\leqslant i\leqslant n$, we let $$\Lambda_{\xi, n, i}=\bigl\{(\omega^\xi(n-1)+t_1)\smallfrown \ldots \smallfrown (\omega^\xi (n-i)+t_i):  t_i\in \Gamma_{\xi,1}, t_1, \ldots, t_{i-1}\in MAX(\Gamma_{\xi,1})\bigr\}.$$ We refer to the sets $\Lambda_{\xi,n,1}, \ldots, \Lambda_{\xi,n,n}$ as the \emph{levels} of $\Gamma_{\xi,n}$.

For a directed set $D$, an ordinal $\xi$, and $n\in \nn$, we let $$\Lambda_{\xi,n,i}.D = \{(\zeta_j, u_j)_{j=1}^k: (\zeta_i)_{i=1}^k\in \Lambda_{\xi,n}, u_i\in D\}.$$   We remark that for each $\zeta$, then for any directed set $D$, $(\omega^\zeta+\Gamma_{\zeta+1,1}).D$ is canonically identifiable with $\Gamma_{\zeta+1, 1}.D$. For any $\xi$ and any $n\in \nn$, $\Lambda_{\xi, n, 1}.D$ is canonically identifiable with $\Gamma_{\xi,1}.D$.  We often implicitly use these canonical identifications without giving them specific names. 

We last define what it means for a subset of $\Gamma_{\xi,n}.D$ to be a \emph{unit}. For any ordinal $\xi$ and any $n\in \nn$, $\Lambda_{\xi,n,1}.D$ is a unit.  If for some $n\in \nn$, every ordinal $\xi$, and every $1\leqslant k\leqslant n$, the units in $\Gamma_{\xi,k}.D$ are defined, we say a subset $U$ of $\Gamma_{\xi, n+1}.D$ is a unit if either $U=\Lambda_{\xi, n+1, 1}.D$ or if there exists $t\in MAX(\Lambda_{\xi, n+1, 1}.D)$ such that, if $$j:\{s\in \Gamma_{\xi,n+1}.D: t<s\}\to \Gamma_{\xi,n}.D$$ is the canonical identification, $j(U)$ is a unit in $\Gamma_{\xi,n}.D$

\subsection{Cofinal and eventual sets}

For a fixed directed set $D$, we now define sets $\Omega_{\xi, n}$. Each set $\Omega_{\xi,n}$ will be a subset of the power set of $MAX(\Gamma_{\xi,n}.D)$.  Given $\mathcal{E}\subset \Gamma_{0,1}.D$, we can write $$\mathcal{E}=\{(0, u): u\in D_0\}$$ for some $D_0\subset D$. Then we say $\mathcal{E}\in \Omega_{0,1}$ if  $D_0$ is cofinal in $D$. 

Now suppose that for a limit ordinal $\xi$ and every $\zeta<\xi$, $\Omega_{\zeta+1, 1}$ has been defined. For each $\zeta<\xi$, let $j_\zeta:(\omega^\zeta+\Gamma_{\zeta+1}).D\to \Gamma_{\zeta+1,1}.D$ be the canonical identification. Then a subset $\mathcal{E}\subset MAX(\Gamma_{\xi, 1})$ lies in $\Omega_{\xi, 1}$ if  there exists a cofinal subset $M$ of $[0, \xi)$ such that for every $\zeta\in M$, $j_\zeta(\mathcal{E}\cap MAX((\omega^\zeta+\Gamma_{\zeta+1}.D))\in \Omega_{\zeta+1, 1}$. 

Now suppose that for an ordinal $\xi$ and every $n\in \nn$, $\Omega_{\xi,n}$ has been defined. Then we say $\mathcal{E}\subset MAX(\Gamma_{\xi+1, 1}.D)$ is a member of $\Omega_{\xi+1, 1}$ if there exists a cofinal subset $M$ of $\nn$ such that for every $n\in \nn$, $\mathcal{E}\cap \Gamma_{\xi, n}.D\in \Omega_{\xi,n}$. 

Last, suppose that for an ordinal $\xi$, a natural number $n$, and each $1\leqslant i\leqslant n$, $\Omega_{\xi,i}$ has been defined. Suppose that $\mathcal{E}\subset MAX(\Gamma_{\xi,n+1}.D)$ is given.   For each $t\in MAX(\Lambda_{\xi,n,1})$, let $P_t=\{s\in \Gamma_{\xi, n+1}.D: t<s\}$,  let $j_t:P_t\to \Gamma_{\xi,n}.D$ be the canonical identification, and let $j:\Lambda_{\xi, n,1}.D\to \Gamma_{\xi,1}.D$ be the canonical identification.  Let $$\mathcal{F}=\{t\in MAX(\Lambda_{\xi,n+1,1}.D): j_t(\mathcal{E}\cap MAX(P_t))\in \Omega_{\xi,n}\}.$$  Then we say $\mathcal{E}\in \Omega_{\xi, n+1}$ if $j(\mathcal{F})\in \Omega_{\xi,1}$.  

We remark that an easy induction proof shows that $MAX(\Gamma_{\xi,n}.D)\in \Omega_{\xi,n}$ for every $\xi$ an $n$, and if $\mathcal{F}\subset \mathcal{E}\subset MAX(\Gamma_{\xi, n}.D)$ and $\mathcal{F}\in \Omega_{\xi,n}$, then $\mathcal{E}\in \Omega_{\xi,n}$.

We refer to the sets in $\Omega_{\xi,n}$ as \emph{cofinal in} $\Gamma_{\xi,n}.D$.   We say a subset $\mathcal{E}$ of $MAX(\Gamma_{\xi,n}.D)$ is \emph{eventual} if $MAX(\Gamma_{\xi,n}.D)\setminus \mathcal{E}$ fails to be cofinal.  Each unit $U\subset \Gamma_{\xi,n}.D$ is canonically identifiable with  $\Gamma_{\xi,1}.D$, and as such we can define what it means for a subset of $MAX(U)$ to be cofinal or eventual using the identification with $\Gamma_{\xi,1}.D$.

We next recall some results from \cite{AUF}.   For the following results, we say $d:\Gamma_{\xi,n}.D\to \Gamma_{\xi,n}.D$ is a \emph{level map} if \begin{enumerate}[(i)]\item for any $\varnothing<s<t\in \Gamma_{\xi,n}.D$, $d(s)<d(t)$, \item if $U\subset \Lambda_{\xi,n,i}.D$ is a unit, then there exists a unit $V\subset \Lambda_{\xi,n,i}.D$ such that $d(U)\subset V$. \end{enumerate}  Note that since $\Gamma_{\xi,1}.D$ is a single unit, $(ii)$ is vacuous in the case $n=1$. Given a level map $d:\Gamma_{\xi,n}.D\to \Gamma_{\xi,n}.D$, we say $e:MAX(\Gamma_{\xi,n}.D)\to MAX(\Gamma_{\xi,n}.D)$ is an \emph{extension} of $d$ if for any $t\in MAX(\Gamma_{\xi,n}.D)$, $d(t)\leqslant e(t)$. Since $\Gamma_{\xi,n}.D$ is well-founded, any level map $d$ admits some extension.  We define an extension of a monotone map in the same way we define an extension of a level map. 

We let $\Pi(\Gamma_{\xi,n}.D)=\{(s,t)\in \Gamma_{\xi,n}.D\times MAX(\Gamma_{\xi,n}.D): s\leqslant t\}$.

\begin{lemma} Suppose that $\xi$ is an ordinal, $n\in \nn$, $X$ is a Banach space, and $(x_t)_{t\in \Gamma_{\xi,n}.D}$ is weakly null.  \begin{enumerate}[(i)]\item If $\mathcal{E}\subset MAX(\Gamma_{\xi,n}.D)$ is cofinal, there exists a level map $d:\Gamma_{\xi,n}.D\to \Gamma_{\xi,n}.D$ with extension $e$ such that $e(MAX(\Gamma_{\xi,n}.D))\subset \mathcal{E}$ and $(x_{d(t)})_{t\in \Gamma_{\xi,n}.D}$ is weakly null.   \item For any $k\in \nn$, if $MAX(\Gamma_{\xi,n}.D)\supset \mathcal{E}=\cup_{i=1}^k \mathcal{E}_i\in \Omega_{\xi,n}$, then there exists $1\leqslant j\leqslant k$ such that $\mathcal{E}_j\in \Omega_{\xi,n}$. \item If $F$ is a finite set and $\chi:\Pi(\Gamma_{\xi,n}.D)\to F$ is a function, then there exist a level map $d:\Gamma_{\xi,n}.D\to \Gamma_{\xi,n}.D$ with extension $e$ and $\alpha_1, \ldots, \alpha_n\in F$ such that for any $1\leqslant i\leqslant n$ and any $\Lambda_{\xi,n,i}.D\ni s\leqslant \in MAX(\Gamma_{\xi,n}.D)$, $\alpha_i=F(d(s), e(t))$, and such that $(x_{d(t)})_{t\in \Gamma_{\xi,n}.D}$ is weakly null.   \item If $h:\Pi(\Gamma_{\xi, n}.D)\to \rr$ is bounded and if $\mathcal{E}\subset MAX(\Gamma_{\xi,n}.D)$ is cofinal, then for any $\delta>0$,  there exist $a_1, \ldots, a_n\in \rr$ and a level map $d:\Gamma_{\xi,n}.D\to \Gamma_{\xi,n}.D$ with extension $e$ such that $e(MAX(\Gamma_{\xi,n}.D))\subset \mathcal{E}$, for each $1\leqslant i\leqslant n$ and each $\Lambda_{\xi,n,i}.D\ni s\leqslant t\in MAX(\Gamma_{\xi,n}.D)$, $h(d(s), e(t))\geqslant a_i-\delta$, and for any $t\in MAX(\Gamma_{\xi,n}.D)$, $\sum_{\varnothing<s\leqslant e(t)} \mathbb{P}_{\xi,n}(s)h(s, e(t)) \leqslant \delta+\sum_{i=1}^n a_i$.   \end{enumerate}

\label{stabilize}

\end{lemma}

\begin{rem}\upshape Items $(i)$ and $(ii)$ together yield that if $MAX(\Gamma_{\xi,n}.D)=\cup_{i=1}^k \mathcal{E}_i$, then there exist $1\leqslant j\leqslant k$ and a level map $d:\Gamma_{\xi,n}.D\to \Gamma_{\xi,n}.D$ with extension $e$ such that $(x_{d(t)})_{t\in \Gamma_{\xi,n}.D}$ is weakly null and $e(MAX(\Gamma_{\xi,n}.D))\subset \mathcal{E}_j$. A typical application of this result will be to have a real-valued function $h:MAX(\Gamma_{\xi,n}.D)\to C\subset \rr$, where $C$ is compact.  We may then fix $\delta>0$ and a finite cover $F_1, \ldots, F_k$ of $C$ by sets of diameter less than $\delta$. We then let $\mathcal{E}_i$ denote those $t\in MAX(\Gamma_{\xi,n}.D)$ such that $h(t)\in F_i$.  We may then find $d$, $e$, and $j$ as above and obtain $(x_{d(t)})_{t\in \Gamma_{\xi,n}.D}$ weakly null such that for every $t\in MAX(\Gamma_{\xi,1}.D)$, $h(e(t))\in F_j$.

Similarly, we will often apply $(iii)$ to a function $h_1:\Pi(\Gamma_{\xi,n}.D)\to C\subset \rr$, where $C$ is compact, by first covering $C$ by $F_1, \ldots, F_k$ of sets of diameter less than $\delta$. We then define $h(s,t)$ to be the minimum $j\leqslant k$ such that $h_1(s,t)\in F_j$.

\end{rem}

\section{Szlenk index}

Given a Banach space $X$,  a weak$^*$-compact subset $K$ of $X^*$, and $\ee>0$, we let $s_\ee(K)$ denote the set of those $x^*\in K$ such that for every weak$^*$-neighborhood $V$ of $x^*$, $\text{diam}(V\cap K)>\ee$.  We then define the transfinite derivations by $$s^0_\ee(K)=K,$$ $$s^{\xi+1}_\ee(K)= s_\ee(s^\xi_\ee(K)),$$ and if $\xi$ is a limit ordinal, $$s^\xi_\ee(K)=\bigcap_{\zeta<\xi} s^\zeta_\ee(K).$$ For convenience, we define $s_\ee^\xi(K)=K$ for each $\ee\leqslant 0$.  We let $Sz(K, \ee)$ be the minimum $\xi$ such that $s^\xi_\ee(K)=\varnothing$, assuming such an ordinal exists.  If no such ordinal exists, we write $Sz(K, \ee)=\infty$.   We let $Sz(K)=\sup_{\ee>0} Sz(K, \ee)$, with the agreement that $Sz(K)=\infty$ if $Sz(K, \ee)=\infty$ for some $\ee>0$.    Given an operator $A:X\to Y$, we let $Sz(A, \ee)=Sz(A^*B_{Y^*}, \ee)$ and $Sz(A)= Sz(A^*B_{Y^*})$.  Given a Banach space $X$, we let $Sz(X, \ee)=Sz(I_X, \ee)$ and $Sz(X)=Sz(I_X)$.

Given an ordinal $\xi$, $1\leqslant q<\infty$, and $M\geqslant 0$,  we say $K$ has $M$-$\xi$-$q$-\emph{summable Szlenk index} if for any $\ee_1, \ldots, \ee_n\geqslant 0$ such that $s^{\omega^\xi}_{\ee_1}\ldots s^{\omega^\xi}_{\ee_n}(K)\neq \varnothing$, $\sum_{i=1}^n \ee_i^q\leqslant M^q$. We say $K$ has $\xi$-$q$-\emph{summable Szlenk index} if it has $M$-$\xi$-$q$-summable Szlenk index for some $M\geqslant 0$.  Given an operator $A:X\to Y$, we say $A$ has $\xi$-$q$-summable Szlenk index if $A^*B_{Y^*}$ does, and we say the Banach space $X$ has $\xi$-summable Szlenk index if $I_X$ does.   The notion of $0$-$1$-summable Szlenk index has been previously defined in \cite{GKL}, and it is quite important to the non-linear theory of Banach spaces and renorming theory. For $\xi>0$ or $1<q<\infty$,  the notion of $\xi$-$q$-summable Szlenk index is new.

Suppose $X$ is a Banach space and $K\subset X^*$ is weak$^*$-compact and non-empty. If $Sz(K)\leqslant \omega^{\xi+1}$, then by weak$^*$-compactness, for every $\ee>0$, $Sz(K, \ee)\leqslant \omega^\xi n$ for some $n\in \nn$.   We may then let $Sz_\xi(K, \ee)$ be the smallest $n\in \nn$ such that $Sz(K, \ee)\leqslant \omega^\xi n$.    For this $K$, we define the $\xi$-\emph{Szlenk power type} $\textbf{p}_\xi(K)$ by $$\textbf{p}_\xi(K):=\underset{\ee\to 0^+}{\lim\sup} \frac{\log Sz(K, \ee)}{|\log(\ee)|}.$$  This value need not be finite. We note that if $Sz(K)\leqslant \omega^\xi$, $Sz_\xi(K, \ee)=1$ for all $\ee>0$, whence $\textbf{p}_\xi(K)=0$.For completeness, we write $\textbf{p}_\xi(K)=\infty$ if $Sz(K)>\omega^{\xi+1}$.

We remark that if $K$ has $M$-$\xi$-$q$-summable Szlenk index, then $\textbf{p}_\xi(K)\leqslant q$.  Indeed, for any $\ee>0$, if $n<Sz_\xi(K, \ee)$, $$\varnothing\neq s^{\omega^\xi n}_\ee(K)= s^{\omega^\xi}_\ee\ldots s^{\omega^\xi}_\ee(K),$$ and $\ee^q n=\sum_{i=1}^n \ee_i^q \leqslant M^q$. From this it follows that $$\log Sz_\xi(K, \ee) \leqslant \log(1+M^q/\ee^q),$$ whence $$\underset{\ee\to 0^+}{\lim\sup} \frac{\log Sz_\xi(K, \ee)}{|\log(\ee)|}\leqslant q.$$

We also recall the following fact.  This fact was shown by Lancien \cite{survey} in the case that $K=B_{X^*}$. 

\begin{proposition} \cite[Lemma $3.8$]{C3} Let $X$ be a Banach space, $K\subset X^*$ weak$^*$-compact and convex, $\xi$ an ordinal. \begin{enumerate}[(i)]\item If for some $\ee>0$,  $Sz(K, \ee)>\omega^\xi$, then $Sz(K, \ee/n)> \omega^\xi n$. \item If $\omega^\xi<Sz(K)\leqslant \omega^{\xi+1}$, then $Sz(K)=\omega^{\xi+1}$. \item If $Sz(K)>\omega^\xi$ then $\textbf{\emph{p}}_\xi(K)\geqslant 1$. \end{enumerate}

\label{convex body1}
\end{proposition}

\begin{corollary} If $X$ is a Banach space and$K\subset X^*$ is a weak$^*$-compact, convex, non-empty set,  then either $Sz(K)=\infty$ or  there exists an ordinal $\xi$ such that $Sz(K)=\omega^\xi$.

\label{convex body}

\end{corollary}

\section{The quantities $\alpha_{\xi,p,n}(K)$}

Given a directed set $D$ and a collection $(x_t)_{t\in \Gamma_{\xi,n}.D}\subset X$, if $t\in MAX(\Gamma_{\xi,n}.D)$, there exist $\varnothing=t_0<\ldots <t_n=t$ such that $t_i\in MAX(\Lambda_{\xi, n,i}.D)$ for each $1\leqslant i\leqslant n$. We then let $$z^t_i=\sum_{t_{i-1}<s\leqslant t_i} \mathbb{P}_{\xi,n}(s)x_s \in \text{co}(x_s: t_{i-1}<s\leqslant t_i).$$   This notation should reference the underlying collection $(x_t)_{t\in \Gamma_{\xi,n}.D}$, but the notation will not cause confusion.

Let $X$ be a Banach space and let $K\subset X^*$ be weak$^*$-compact.    For $x\in X$, let $r_K(x)=0$ if $K=\varnothing$, and otherwise let $r_K(x)=\max_{x^*\in K}\text{Re\ }x^*(x)$.  For an ordinal $\xi$, $1\leqslant p\leqslant \infty$, and $n\in \nn$, we let $\alpha_{\xi,p,n}(K)$ be the infimum of those $\alpha>0$ such that for any directed set $D$, any $(a_i)_{i=1}^n \in \mathbb{K}^n$,  and any weakly null $(x_t)_{t\in \Gamma_{\xi, n}.D}\subset B_X$, $$\inf_{t\in MAX(\Gamma_{\xi,n}.D)} r_K(\sum_{i=1}^n a_i z^t_i) \leqslant \alpha \|(a_i)_{i=1}^n \|_{\ell_p^n}.$$ We let $\alpha_{\xi,p}(K)=\sup_n \alpha_{\xi,p,n}(K)$.   Let $\theta_{\xi,n}(K)$ be the infimum of those $\theta>0$ such that for any directed set $D$ and any weakly null $(x_t)_{t\in \Gamma_{\xi,n}.D}\subset B_X$, $$\inf_{t\in MAX(\Gamma_{\xi,n}.D)} r_K(\sum_{i=1}^n \frac{1}{n}z^t_i) \leqslant \theta.$$

\begin{rem}\upshape It is an easy consequence of  \cite[Theorem $2.2$]{CAlt} is that if $X$ is a Banach space, $K\subset X^*$ is weak$^*$-compact, $\xi$ is an ordinal, $n\in \nn$, $\ee>0$, $Sz(K, \ee)>\omega^\xi n$, and $D$ is any weak neighborhood basis at $0$ in $X$, then for any $0<\delta<\ee/4$, there exist a weakly null collection $(x_t)_{t\in \Gamma_{\xi,n}.D}\subset B_X$ and $(x^*_t)_{t\in MAX(\Gamma_{\xi,n}.D)}\subset K$ such that for each $\varnothing<s\leqslant t\in MAX(\Gamma_{\xi,n}.D)$, $\text{Re\ }x^*_t(x_s)\geqslant \delta$. In particular, $\theta_{\xi, n}(K)\geqslant \ee/4$. 

Conversely, it follows from \cite[Theorem $2.2$]{CAlt} and \cite[Corollary $5.3$]{C} that if $\theta_{\xi,1}(K)\geqslant \ee$, $s^{\omega^\xi}_{\ee_0}(K)\neq \varnothing$ for any $0<\ee_0<\ee$.

\label{alternative}
\end{rem}

\begin{rem}\upshape For later convenience, the definition of $\alpha_{\xi,p,n}$ considers all weakly null collections indexed by any directed set $D$. However, in the definition of $\alpha_{\xi,p,n}$ could be taken to include only weakly null collections indexed by $\Gamma_{\xi,n}.D_1$, where $D_1$ is some fixed weak neighborhood basis at $0$ in $X$. We will freely use this fact throughout. In order to see why this holds, fix a scalar sequence $(a_i)_{i=1}^n$ and a positive number $\alpha$. If we have $$\inf_{t\in MAX(\Gamma_{\xi,n}.D)} r_K(\sum_{i=1}^n a_i z^t_i)\leqslant \alpha \|(a_i)_{i=1}^n\|_{\ell_p^n}$$ for every directed set $D$ and every weakly null collection $(x_t)_{t\in \Gamma_{\xi,n}.D}\subset B_X$, then we obviously have it whenever $D=D_1$.   For the converse, if there exists a directed set $D$ and a weakly null collection $(x_t)_{t\in \Gamma_{\xi,n}.D}\subset B_X$ such that $\inf_{t\in MAX(\Gamma_{\xi,n}.D)}r_K(\sum_{i=1}^n a_i z^t_i)>\alpha \|(a_i)_{i=1}^n\|_{\ell_p^n}$, we may define a map $\phi:\Gamma_{\xi,n}.D_1\to \Gamma_{\xi,n}.D$ such that, with $u_t=x_{\phi(t)}$, $(u_t)_{t\in \Gamma_{\xi,n}.D_1}\subset B_X$ is weakly null and for every $t\in MAX(\Gamma_{\xi,n}.D_1)$, $$\sum_{i=1}^n a_i\sum_{\Lambda_{\xi,n,i}.D_1} \mathbb{P}_{\xi,n}(s)u_s= \sum_{i=1}^n a_i z^{\phi(t)}_i,$$ whence $$\inf_{t\in MAX(\Gamma_{\xi,n}.D_1)} r_K(\sum_{i=1}^n a_i\sum_{\Lambda_{\xi,n,i}.D_1} \mathbb{P}_{\xi,n}(s)u_s) \inf_{t\in MAX(\Gamma_{\xi,n}.D)}r_K(\sum_{i=1}^n a_i z^t_i)>\alpha \|(a_i)_{i=1}^n\|_{\ell_p^n}.$$

\label{super remark}
\end{rem}

The previous remark will be useful, for example, when considering direct sums.  If $A_1:X_1\to Y_1$ and $A_2:X_2\to Y_2$ are operators, we may wish to consider $\alpha_{\xi,p, n}(A_1\oplus A_2: X_1\oplus_r X_2\to Y_1\oplus_r Y_2)$. If $\alpha_1<\alpha_{\xi,p,n}(A_1)$, we will find a directed set $D_1$ and a weakly null collection $(x_t^1)_{t\in \Gamma_{\xi, n}.D_1}\subset B_{X_1}$ to witness that $\alpha_1<\alpha_{\xi,p,n}(A_1)$. Similarly, if $\alpha_2<\alpha_{\xi,p,n}(A_2)$, we will find a directed set $D_2$ and a weakly null collection $(x^2_t)_{t\in \Gamma_{\xi, n}.D_2}\subset B_{X_2}$ to witness that $\alpha_2<\alpha_{\xi, p,n}(A_2)$.  By the definition, we do not have control over $D_1$ or $D_2$, and the fact that $D_1$ need not be equal to $D_2$ is problematic. We will use Remark \ref{super remark} to deduce that, in both cases, we can take $D_1=D_2=D$, where $D$ is a fixed weak neighborhood basis at $0$ in $X_1\oplus_r X_2$.

One inconvenience of the definitions of $\alpha_{\xi, p,n}$ and $\theta_{\xi,n}$ is that they involve special convex combinations, and the convex coefficient on a vector depends upon its position in the tree.  Therefore even if we know that $$\alpha<\inf_{t\in MAX(\Gamma_{\xi,n}.D)} r_K(\sum_{i=1}^n a_i z_i^t),$$  if $d:\Gamma_{\xi,n}.D\to \Gamma_{\xi,n}.D$ is a level map such that $(x_{d(t)})_{t\in \Gamma_{\xi,n}.D}$ is weakly null, we do not know that $$\alpha < \inf_{t\in MAX(\Gamma_{\xi,n}.D)} r_K(\sum_{i=1}^n \sum_{t\geqslant s\in \Lambda_{\xi,n, i}.D} a_i \mathbb{P}_{\xi,n}(s) x_{d(x)}).$$  Initially this prevents us from using our combinatorial lemmas to stabilize certain quantities for members of a given tree $(x_t)_{t\in \Gamma_{\xi,n}.D}$ which was chosen to witness that $\alpha<\alpha_{\xi,p,n}(K)$.   But suppose that $(a_i)_{i=1}^n$ are non-negative reals and that we have numbers $b_1, \ldots, b_n$ such that $\sum_{i=1}^n a_i b_i>\alpha$, a weakly null collection $(x_t)_{t\in \Gamma_{\xi,n}.D}\subset B_X$, and a collection $(x^*_t)_{t\in MAX(\Gamma_{\xi,n}.D)}\subset K$ such that for every $1\leqslant i\leqslant n$ and $\Lambda_{\xi,n,i}.D\ni s\leqslant t\in MAX(\Gamma_{\xi,n}.D)$, $\text{Re\ }x^*_t(x_s)\geqslant b_i$.   Then for any level map $d:\Gamma_{\xi, n}.D\to \Gamma_{\xi,n}.D$, any extension $e$ of $d$, any $1\leqslant i\leqslant n$, and $\Lambda_{\xi,n,i}.D\ni s\leqslant t\in MAX(\Gamma_{\xi,n}.D)$, $\text{Re\ }x^*_{e(t)}(x_{d(s)})\geqslant b_i$.    Therefore for any $t\in MAX(\Gamma_{\xi,n}.D)$, $$r_K(\sum_{i=1}^n \sum_{e(t)\geqslant s\in \Lambda_{\xi,n,i}.D } a_i \mathbb{P}_{\xi,n}(s) x_{d(s)}) \geqslant x^*_{e(t)}(\sum_{i=1}^n \sum_{e(t)\geqslant s\in \Lambda_{\xi,n,i}.D } a_i \mathbb{P}_{\xi,n}(s) x_{d(s)}) \geqslant \sum_{i=1}^n a_i b_i>\alpha.$$  Therefore if we have a map $f:B_X\to M$ into a compact metric space $M$, we may apply Lemma \ref{stabilize} to the function $F:\Pi(\Gamma_{\xi,n}.D)\to M$ given by $F(s,t)=f(x_s)$ to deduce that for any $\delta>0$, there exist $\varpi_1, \ldots, \varpi_n\in M$,  a level map $d:\Gamma_{\xi,n}.D\to \Gamma_{\xi,n}.D$, and an extension $e$ of $d$ such that $d_M(F(d(s), e(t)), \varpi_i)=d_M(f(x_{d(s)}), \varpi_i)<\delta$ for each $1\leqslant i\leqslant n$ and each $\Lambda_{\xi,n,i}.D\ni s\leqslant t\in MAX(\Gamma_{\xi,n}.D)$, and the collections $(x_{d(t)})_{t\in \Gamma_{\xi,n}.D}$, $(x^*_{e(t)})_{t\in MAX(\Gamma_{\xi,n}.D)}\subset K$ can be used as described above to deduce that  $$\alpha<\inf_{t\in MAX(\Gamma_{\xi,n}.D)} r_K(\sum_{i=1}^n \sum_{e(t)\geqslant s\in \Lambda_{\xi,n,i}.D} a_i \mathbb{P}_{\xi,n}(s) x_{d(s)}).$$   Thus we can replace the collection $(x_t)_{t\in \Gamma_{\xi,n}.D}$ with $(x_{d(t)})_{t\in \Gamma_{\xi,n}.D}$ without losing the inequality coming from the definition of $\alpha_{\xi,p,n}$ or $\theta_{\xi,n}$ while stabilizing the function $f$ on each level of the collection.  This is the primary motivation for the next theorem.

\begin{theorem} Let $X$ be a Banach space, $K\subset X^*$ weak$^*$-compact, $1\leqslant p\leqslant \infty$. \begin{enumerate}[(i)]\item For $\alpha \in \rr$, $\alpha_{\xi,p,n}(K)>\alpha$ if and only if there exist non-negative scalars $(a_i)_{i=1}^n\in B_{\ell_p^n}$, a directed set $D$, a weakly null collection $(x_t)_{t\in \Gamma_{\xi,n}.D}\subset B_X$, and $\alpha'>\alpha$ such that $$\bigl\{ t\in MAX(\Gamma_{\xi,n}.D): r_K(\sum_{i=1}^n a_i z^t_i)>\alpha'\bigr\}$$ is cofinal if and only if there exist non-negative scalars $(a_i)_{i=1}^n\in B_{\ell_p^n}$, a directed set $D$, a weakly null collection $(x_t)_{t\in \Gamma_{\xi,n}.D}\subset B_X$, $(x^*_t)_{t\in MAX(\Gamma_{\xi,n}.D)}\subset K$, and  non-negative real numbers $b_1, \ldots, b_n$ such that $\alpha<\sum_{i=1}^n a_ib_i$ and for each $1\leqslant i\leqslant n$ and each $\Lambda_{\xi,n,i}.D\ni s\leqslant t\in MAX(\Gamma_{\xi,n}.D)$, $\text{\emph{Re}\ }x^*_t(x_s) \geqslant b_i$. \item $\alpha_{\xi,p,n}(K)=0$ for some $n\in \nn$ if and only if $\alpha_{\xi,p,n}(K)=0$ for all $n\in \nn$ if and only if $Sz(K)\leqslant \omega^\xi$. \item If $R\geqslant 0$ is such that $K\subset RB_{X^*}$, $\alpha_{\xi, p,n}(K)\leqslant n^{1-1/p}R$. \item $\alpha_{\xi,p,n}(K)= \alpha_{\xi,p,n}(\overline{\text{\emph{abs\ co}}}^{\text{\emph{weak}}^*}(K))$.   \end{enumerate}

\label{big theorem}

\end{theorem}

\begin{proof}$(i)$ First assume there exist $\alpha'>\alpha$, a directed set $D$,  non-negative scalars $(a_i)_{i=1}^n\in B_{\ell_p^n}$, and a weakly null collection $(x_t)_{t\in \Gamma_{\xi,n}.D}\subset B_X$ such that $$\mathcal{E}:=\{t\in MAX(\Gamma_{\xi,n}.D): r_K(\sum_{i=1}^n a_i z^t_i)>\alpha'\}$$ is cofinal. For every $t\in MAX(\Gamma_{\xi,n}.D)$, fix $x^*_t\in K$ such that $$\text{Re\ }x^*_t(\sum_{i=1}^n a_i z^t_i)=r_K(\sum_{i=1}^n a_i z^t_i).$$  Define $h:\Pi(\Gamma_{\xi,n}.D)\to \rr$ by $h(s,t)=\text{Re\ }x^*_t(a_i x_s)$, where $\Lambda_{\xi, n,i}\ni s\leqslant t\in MAX(\Gamma_{\xi,n}.D)$.   Fix a positive number $\delta$ such that $\delta(n+1)<\alpha'-\alpha$.  By Lemma \ref{stabilize}$(iv)$,   there exists a level map $d:\Gamma_{\xi,n}.D\to \Gamma_{\xi,n}.D$ with extension $e:MAX(\Gamma_{\xi,n}.D)\to \mathcal{E}$ and scalars $b_1', \ldots, b_n'\in \rr$ such that $(x_{d(t)})_{t\in \Gamma_{\xi,n}.D}$ is weakly null, for each $1\leqslant i\leqslant n$ and each $\Lambda_{\xi,n,i}.D\ni s\leqslant t\in MAX(\Gamma_{\xi,n}.D)$, $\text{Re\ }x^*_{e(t)}(a_i x_{d(t)}) \geqslant b'_i-\delta$, and for each $t\in MAX(\Gamma_{\xi,n}.D)$,  $$\alpha'< x^*_{e(t)}\bigl(\sum_{i=1}^n a_i z^{d(t)}_i\bigr) =\sum_{\varnothing<s\leqslant e(t)} \mathbb{P}_{\xi,n}(s)h(s, e(t))<\delta+\sum_{i=1}^n b_i'.$$  By relabeling, we assume $d(t)=t$ and $e(t)=t$. Furthermore, by replacing $x_t $ with $0$ for any $1\leqslant i\leqslant n$ such that $b'_i-\delta\leqslant 0$ and $t\in \Lambda_{\xi, n, i}.D$, we may assume $x_t=0$ for any such $t$ and $i$.  We then let $b_i=0$ provided $b_i'-\delta\leqslant 0$, and $b_i=(b_i'-\delta)/a_i$ if $b_i'-\delta>0$.  Note that the condition $b_i'-\delta>0$ implies $b_i>0$.  Then \begin{align*} \inf_{t\in MAX(\Gamma_{\xi,n}.D)} r_K(\sum_{i=1}^n a_i z^t_i) &  \geqslant \text{Re\ }x^*_t(\sum_{i=1}^n a_i z^t_i) \geqslant \sum_{i=1}^n \sum_{\Lambda_{\xi,n,i}\ni s\leqslant t} \mathbb{P}_{\xi,n}(s) \text{Re\ }x^*_t(a_i x_s) \\ & \geqslant \sum_{i=1}^n (b_i'-\delta) \geqslant \alpha'-(n+1)\delta>\alpha.\end{align*}

Now assume there exist a directed set $D$, non-negative scalars $(a_i)_{i=1}^n\in B_{\ell_p^n}$, and a weakly null collection $(x_t)_{t\in \Gamma_{\xi,n}.D}\subset B_X$ such that $$\inf_{t\in \Gamma_{\xi,n}.D} r_K(\sum_{i=1}^n a_i z^t_i)>\alpha.$$  Then for any $\alpha'$ such that $$\inf_{t\in \Gamma_{\xi,n}.D} r_K(\sum_{i=1}^n a_i z^t_i) >\alpha'>\alpha,$$  $$\{t\in MAX(\Gamma_{\xi,n}.D): r_K(\sum_{i=1}^n a_i z^t_i)>\alpha'\}=MAX(\Gamma_{\xi,n}.D)$$ is cofinal.

Obviously the condition in the previous paragraph implies that $\alpha_{\xi,p,n}(K)>\alpha$.   Conversely, suppose there exists $(a_i)_{i=1}^n\in \mathbb{K}^n$, a directed set $D$, and a weakly null collection $(x_t)_{t\in \Gamma_{\xi,n}.D}\subset B_X$ such that $$\inf_{t\in \Gamma_{\xi,n}.D} r_K(\sum_{i=1}^n a_i z^t_i)>\alpha \|(a_i)_{i=1}^n\|_{\ell_p^n}.$$  By positive homogeneity of $r_K$, we may assume $(a_i)_{i=1}^n\in S_{\ell_p^n}$.    For each $1\leqslant i\leqslant n$, fix a unimodular scalar $\ee_i$ such that $a_i/\ee_i=|a_i|$.  For each $1\leqslant i\leqslant n$ and each $t\in \Lambda_{\xi,n,i}.D$, let $x^\ee_i:= \ee_i x_t$.  For $t\in MAX(\Gamma_{\xi,n}.D)$, let  $$z^{\ee ,t}_i=\sum_{\Lambda_{\xi,n,i}.D\ni s\leqslant t} \mathbb{P}_{\xi,n}(s) x^\ee_s.$$   Then for each $t\in MAX(\Gamma_{\xi,n}.D)$, $$r_K(\sum_{i=1}^n a_i z^t_i) = r_K(\sum_{i=1}^n |a_i| z^{\ee,t}_i).$$  Since $(x^\ee_t)_{t\in \Gamma_{\xi,n}.D}\subset B_X$ is weakly null, this finishes $(i)$.

$(ii)$ The main theorem of \cite{CAlt} can be stated as: $\alpha_{\xi,p,1}(K)=0$ if and only if $Sz(A)\leqslant \omega^\xi$. The rest of $(ii)$ follows from the fact that that $(\alpha_{\xi,p,n}(K))_{n=1}^\infty$ is non-decreasing and the obvious fact that $\alpha_{\xi,p,n}(K)\leqslant n \alpha_{\xi,p,1}(K)$.

$(iii)$ This is obvious from H\"{o}lder's theorem. 

$(iv)$ We let $\mathbb{T}=\{\ee\in \mathbb{K}: |\ee|=1\}$ and $\mathbb{T}K=\{\ee x^*: \ee\in \mathbb{T}, x^*\in K\}$.   We first show that $\alpha_{\xi,p,n}(K)=\alpha_{\xi,p,n}(\mathbb{T}K)$.    Since $r_K\leqslant r_{\mathbb{T}K}$, $\alpha_{\xi,p,n}(K)\leqslant \alpha_{\xi,p,n}(\mathbb{T}K)$.   Fix $\alpha<\alpha_{\xi,n,p}(\mathbb{T}K)$, non-negative scalars $(a_i)_{i=1}^n\in B_{\ell_p^n}$, $b_1, \ldots, b_n\in \rr$ such that $\alpha<\sum_{i=1}^n a_i b_i$, a directed set $D$, a weakly null collection $(x_t)_{t\in \Gamma_{\xi,n}.D}\subset B_X$, and a collection $(\ee_t x^*_t)_{t\in MAX(\Gamma_{\xi,n}.D)}\subset \mathbb{T}K$ such that for every $1\leqslant i\leqslant n$ and every $\Lambda_{\xi,n,i}.D\ni s\leqslant t\in MAX(\Gamma_{\xi,n}.D)$, $$b_i\leqslant \text{Re\ }\ee_t x^*_t(x_s).$$    We may fix such constants, vectors, and functionals by $(i)$.  Here, the collection $$(\ee_t x^*_t)_{t\in MAX(\Gamma_{\xi,n}.D)}\subset \mathbb{T}K$$ is assumed to be written such that $|\ee_t|=1$ and $x^*_t\in K$ for all $t\in MAX(\Gamma_{\xi,n}.D)$. Fix $R>0$ such that $K\subset RB_{X^*}$ and $\delta>0$ such that $nR\delta +\alpha <\sum_{i=1}^n a_i b_i$. Fix a finite $\delta$-net $(\ee_i)_{i=1}^k$ of $\mathbb{T}$. For each $1\leqslant i\leqslant k$, let $\mathcal{E}_i=\{t\in MAX(\Gamma_{\xi,n}.D): |\ee_i-\ee_t|\leqslant \delta\}$.   By Lemma \ref{stabilize}, we may relabel and assume there exists a single $\ee\in \mathbb{T}$ such that $|\ee-\ee_t|\leqslant \delta$ for all $t\in MAX(\Gamma_{\xi,n}.D)$.  Let $c_i=b_i-R\delta$, so that $$\sum_{i=1}^n a_i c_i \geqslant \sum_{i=1}^n a_i b_i - nR\delta >\alpha.$$   We now note that for each $1\leqslant i\leqslant n$ and each $\Lambda_{\xi,n,i}.D\ni s\leqslant t\in MAX(\Gamma_{\xi,n}.D)$, $$\text{Re\ }x^*_t(\ee x_t) \geqslant \text{Re\ }\ee_t x^*_t(x_t) - |\ee-\ee_t|\|x_t\| \geqslant b_i-R\delta =c_i.$$  We now appeal to $(i)$ to deduce that $\alpha_{\xi,p,n}(K)>\alpha$.   This shows that $\alpha_{\xi,p,n}(K)=\alpha_{\xi,p,n}(\mathbb{T}K)$.

We now note that for any weak$^*$-compact set $L$ of $X^*$, $r_L= r_{\overline{\text{co}}^{\text{weak}^*}(L)}$, whence $$\alpha_{\xi,p,n}(L)=\alpha_{\xi,p,n}(\overline{\text{co}}^{\text{weak}^*}(L)).$$  If $L=\mathbb{T}K$, then $$\overline{\text{abs\ co}}^{\text{weak}^*}(K)= \overline{\text{co}}^{\text{weak}^*}(L),$$ and $$\alpha_{\xi,p,n}(\overline{\text{abs\ co}}^{\text{weak}^*}(K))=\alpha_{\xi,p,n}( \overline{\text{co}}^{\text{weak}^*}(L))=\alpha_{\xi,p,n}(L)=\alpha_{\xi,p,n}(\mathbb{T}K)= \alpha_{\xi,p,n}(K).$$

\end{proof}

\begin{rem}\upshape In Theorem \ref{big theorem}$(i)$, we showed that in order to check the values of $\alpha_{\xi, p,n}(K)$, it is sufficient to check only over sequences $(a_i)_{i=1}^n\in S_{\ell_p^n}$ of non-negative, real scalars. We now note that it is sufficient to check only over positive, real scalars by density of $\{(a_i)_{i=1}^n\in \ell_p^n: a_i>0\}$ in $\{(a_i)_{i=1}^n\in \ell_p^n: a_i\geqslant 0\}$.

\end{rem}

\begin{rem}\upshape We note that a weak$^*$-compact set $K\subset X^*$ is norm compact if and only if $\alpha_{0,p,1}(K)=0$ for some $1\leqslant p\leqslant \infty$ if and only if $\alpha_{0,p,1}(K)=0$ for every $1\leqslant p\leqslant \infty$ if and only if for any weakly null net $(x_\lambda)\subset B_X$, $r_K(x_\lambda)\to 0$. 

\label{daniel}
\end{rem}

The results contained in Theorem \ref{big theorem 2} have similar proofs to those contained in Theorem \ref{big theorem}, so we omit them.

\begin{theorem} Let $X$ be a Banach space, $K\subset X^*$ weak$^*$-compact. \begin{enumerate}[(i)]\item $\theta_{\xi,n}(K)=0$ for some $n\in \nn$ if and only if $\theta_{\xi,n}(K)=0$ for all $n\in \nn$ if and only if $Sz(K)\leqslant \omega^\xi$.  \item If $R\geqslant 0$ is such that $K\subset RB_{X^*}$, $\theta_{\xi,n}(K)\leqslant R$. \item $\theta_{\xi,n}(K)= \theta_{\xi,n}(\overline{\text{\emph{abs\ co}}}^{\text{\emph{weak}}^*}(K))$.  \item For $\theta \in \rr$, $\theta_{\xi,n}(K)>\theta$ if and only if there exist a directed set $D$, a weakly null collection $(x_t)_{t\in \Gamma_{\xi,n}.D}\subset B_X$,  $(x^*_t)_{t\in MAX(\Gamma_{\xi,n}.D)}\subset K$, and real numbers $b_1, \ldots, b_n$ such that $\theta n <\sum_{i=1}^n b_i$ and for each $1\leqslant i\leqslant n$ and each $\Lambda_{\xi,n,i}.D\ni s\leqslant t\in MAX(\Gamma_{\xi,n}.D)$, $\text{\emph{Re}\ }x^*_t(x_s) \geqslant b_i$.  \end{enumerate}

\label{big theorem 2}

\end{theorem}

\begin{proposition} Let $X$ be a Banach space and let $K\subset X^*$ be weak$^*$-compact.  The assignment $\omega^\xi n\mapsto \theta_{\xi, n}(K)$ is continuous and non-increasing from $\{\omega^\xi n: \xi\in \textbf{\emph{Ord}}, n\in \nn\}$ into $\rr$.  That is, if $\omega^\zeta m\leqslant \omega^\xi n$ (which happens if and only if $\zeta\leqslant \xi$ and either $\zeta<\xi$ or $\zeta=\xi$ and $m\leqslant n$), then $\theta_{\xi, n}(K)\leqslant \theta_{\zeta, m}(K)$, and for any ordinal $\xi$, $$\theta_{\xi+1, 1}(K)=\inf\{\theta_{\xi,n}(K): n\in \nn\}$$ and if $\xi$ is a limit ordinal, $$\theta_{\xi,1}(K)= \inf\{\theta_{\zeta, n}(K): \zeta<\xi, n\in \nn\}=\inf\{\theta_{\zeta+1, 1}(K): \zeta<\xi\}.$$

\label{too easy}
\end{proposition}

\begin{proof}  In order to see that $\omega^\xi n\mapsto \theta_{\xi,n}(K)$ is non-increasing, it is sufficient to show that for any ordinal $\xi$ and $n\in \nn$, \begin{enumerate}[(i)]\item $\theta_{\xi, n+1}(K)\leqslant \theta_{\xi,n}(K)$, \item $\theta_{\xi+1, 1}(K)\leqslant \theta_{\xi, n}(K)$, \item for any limit ordinal $\lambda>\xi$, $\theta_{\lambda, 1}(K)\leqslant \theta_{\xi+1, 1}(K)$. \end{enumerate}

Let us show (i). Suppose $\vartheta<\theta_{\xi, n+1}(K)$, a directed set $D$, a weakly null collection $(x_t)_{t\in \Gamma_{\xi, n+1}.D}\subset B_X$, and non-negative numbers $b_1, \ldots, b_{n+1}$ such that $\sum_{i=1}^{n+1} \frac{b_i}{n+1}>\vartheta$ and for each $1\leqslant i\leqslant n+1$ and $\Lambda_{\xi,n,i}.D\ni s\leqslant t\in MAX(\Gamma_{\xi, n+1}.D)$, $\text{Re\ }x^*_t(x_s)\geqslant b_i$.    Now for each $1\leqslant i\leqslant n+1$, let $T_i=\{1, \ldots, n+1\}\setminus \{i\}$.   We note that for any $1\leqslant i\leqslant n+1$, $\sum_{j\in T_i} \frac{b_i}{n}\leqslant \theta_{\xi, n}(K)$.  Indeed, by Lemma \ref{stabilize}, there exists a level map $d:\Gamma_{\xi, n}.D\to \Gamma_{\xi, n+1}.D$ such that $d(\Lambda_{\xi, n, j}.D)\subset \Lambda_{\xi, n+1, j}.D)$ for each $j<i$ and $d(\Lambda_{\xi,n,j}.D)\subset \Lambda_{\xi, n+1, j+1}.D$ for $j\geqslant i$.   Then if $e$ is any extension of $d$, $$\sum_{j\in T_i} \frac{b_j}{n} \leqslant \inf_{t\in MAX(\Gamma_{\xi, n}.D)} \text{Re\ }x^*_{e(t)}(\sum_{j=1}^n \sum_{e(t)\geqslant s\in \Lambda_{\xi, n,j}.D} n^{-1} x_{d(s)}) \leqslant \theta_{\xi,n}(K).$$  Then \begin{align*} \vartheta & < \sum_{j=1}^{n+1} \frac{b_j}{n+1} = \sum_{i=1}^{n+1} \sum_{j\in T_i} \frac{b_j}{(n+1)n} \\ & = \frac{1}{n+1}\sum_{i=1}^{n+1} \sum_{j\in T_i}\frac{b_j}{n} \leqslant \frac{1}{n+1}\sum_{i=1}^{n+1} \theta_{\xi,n}(K) \\ & = \theta_{\xi,n}(K). \end{align*} This yields (i).

Now if $\vartheta<\theta_{\xi+1, 1}(K)$, we may fix a weakly null collection $(x_t)_{t\in \Gamma_{\xi+1, 1}.D}\subset B_X$ such that $$\inf_{t\in MAX(\Gamma_{\xi+1, 1}.D)} r_K(\sum_{s\leqslant t}\mathbb{P}_{\xi+1, 1}(s) x_s)>\vartheta.$$   Now since $\mathbb{P}_{\xi+1,1}(s)= n^{-1}\mathbb{P}_{\xi, n}(s)$ for each $s\in \Gamma_{\xi,n}.D$, \begin{align*} \vartheta & < \inf_{t\in MAX(\Gamma_{\xi+1, 1}.D)} r_K(\sum_{s\leqslant t}\mathbb{P}_{\xi+1, 1}(s) x_s) \leqslant \inf_{t\in MAX(\Gamma_{\xi,n}.D)} r_K(\sum_{s\leqslant t} \mathbb{P}_{\xi+1, 1}(s) x_s) \\ & = \inf_{t\in MAX(\Gamma_{\xi,n}.D)} r_K(\sum_{i=1}^n \frac{1}{n}\sum_{t\geqslant s\in \Lambda_{\xi,n,i}.D} \mathbb{P}_{\xi,n}(s) x_s). \end{align*} Therefore $(x_t)_{t\in MAX(\Gamma_{\xi, n}.D)}$ witnesses that $\vartheta<\theta_{\xi,n}(K)$, which yields (ii).

For (iii), we argue as in (ii), noting that there is a canonical map $j:\Gamma_{\xi+1, 1}.D\to \Gamma_{\lambda, 1}.D$ such that $\mathbb{P}_{\xi+1, 1}(t)= \mathbb{P}_{\lambda, 1}(j(t))$ for any $t\in \Gamma_{\xi+1, 1}.D$.  Thus any collection witnessing that $\vartheta<\theta_{\lambda, 1}(K)$ has a subset witnessing that $\vartheta <\theta_{\xi+1, 1}(K)$.

It follows from what we have already shown that for any ordinal $\xi$, $$\theta_{\xi+1, 1}(K)\leqslant \inf\{\theta_{\xi, n}(K): n\in \nn\}.$$  Now if $\vartheta< \inf\{\theta_{\xi,n}(K): n\in \nn\}$,  for a fixed weak neighborhood basis $D$ in $X$ at $0$ and for each $n\in\nn$, we may fix $(x_t)_{t\in \Gamma_{\xi,n}.D}\subset B_X$ such that $$\vartheta<\inf_{t\in MAX(\Gamma_{\xi,n}.D)} r_K(\sum_{i=1}^n \sum_{t\geqslant s\in \Lambda_{\xi,n,i}.D} \frac{1}{n}\mathbb{P}_{\xi,n}(s) x_s) = \inf_{t\in MAX(\Gamma_{\xi+1,1}.D)\cap \Gamma_{\xi,n}.D} r_K(\sum_{s\leqslant t} \mathbb{P}_{\xi+1, 1}(s) x_s).$$ Then $(x_t)_{t\in \Gamma_{\xi+1, 1}.D}=(x_t)_{t\in \cup_{n=1}^\infty \Gamma_{\xi,n}.D}$ witnesses that $\vartheta\leqslant \theta_{\xi+1, 1}(K)$.   This yields that $$\theta_{\xi+1, 1}(K)= \inf\{\theta_{\xi, n}(K): n\in \nn\}.$$

It also follows from what we have shown that for a limit ordinal $\xi$, $$\theta_{\xi, 1}(K) \leqslant \inf \{\theta_{\zeta+1, 1}(K): \zeta<\xi\}= \inf \{\theta_{\zeta, n}(K): \zeta<\xi, n\in \nn\}.$$  We argue similarly to the previous paragraph to deduce the reverse inequality.  For $\vartheta< \inf\{\theta_{\zeta+1, 1}(K): \zeta<\xi\}$, we may fix a weak neighborhood basis $D$ at $0$ in $X$ and for each $\zeta<\xi$ a weakly null collection $(x'_t)_{t\in \Gamma_{\zeta+1, 1}.D}\subset B_X$ such that $$\vartheta<\inf_{t\in MAX(\Gamma_{\zeta+1, 1}.D)} r_K(\sum_{s\leqslant t} \mathbb{P}_{\zeta+1, 1}(s) x_s').$$  Now if $j_\zeta:(\omega^\zeta+\Gamma_{\zeta+1, 1}).D\to \Gamma_{\zeta+1, 1}.D$ is the canonical identification, $(x_{j(t)})_{t\in \Gamma_{\xi, 1}.D}$ witnesses that $\vartheta\leqslant \theta_{\xi, 1}(K)$.

\end{proof}

\begin{corollary} Let $\xi$ be an ordinal, $1\leqslant p,q\leqslant\infty$ with $1<p$ and $1/p+1/q=1$, $X$ a Banach space, and $K\subset X^*$ weak$^*$-compact.  \begin{enumerate}[(i)]\item $\underset{n}{\ \inf\ } \theta_{\xi,n}(K)= \underset{n}{\ \lim\sup\ } \alpha_{\xi,p,n}(K)/n^{1/q}$. \item $Sz(K)\leqslant \omega^{\xi+1}$ if and only if $\underset{n}{\inf}\text{\ }\theta_{\xi,n}(K)=0$ if and only if $\underset{n}{\lim\sup}\text{\ } \alpha_{\xi,p,n}(K)/n^{1/q}=0$.  \end{enumerate}

\label{p}

\end{corollary}

\begin{proof}$(i)$ First note that $\inf_n \theta_{\xi, n}(K)=\lim_n \theta_{\xi, n}(K)$, since $(\theta_{xi, n}(K))_{n=1}^\infty$ is non-increasing.

Since $\|(1/n)_{i=1}^n\|_{\ell_p^n}= 1/n^{1/q}$, it follows that for any weakly null $(x_t)_{t\in \Gamma_{\xi,n}}\subset B_X$, $$\inf_{t\in MAX(\Gamma_{\xi,n}.D)} r_K(\frac{1}{n}\sum_{i=1}^n z^t_i)\leqslant \alpha_{\xi,p,n}(K)/n^{1/q}.$$  Thus $\theta_{\xi,n}(K)\leqslant \alpha_{\xi,p,n}(K)/n^{1/q}$, and $$\underset{n}{\inf} \theta_{\xi,n}(K)\leqslant \underset{n}{\lim\sup} \alpha_{\xi,p,n}(K).$$  Now to obtain a contradiction, assume $$\underset{n}{\inf} \theta_{\xi,n}(K)<\vartheta< \theta<\underset{n}{\lim\sup} \alpha_{\xi,p,n}(K).$$ Let $R> 0$ be such that $K\subset RB_{X^*}$.  Fix $N_1\in \nn$ such that $\theta_{\xi,n}(K)<\vartheta$ for all $n\geqslant N_1$ and fix $N_2$ so large that $N_1\leqslant (\theta-\vartheta)N_2^{1/q}/R $.   Now fix $n\geqslant N_2$ such that $\alpha_{\xi,p,n}(K)/n^{1/q}>\theta$.   There exist a sequence of  non-negative scalars $(a_i)_{i=1}^n\in B_{\ell_p^n}$, scalars $b_1, \ldots, b_n\in [0,R]$ such that $\theta n^{1/q}< \sum_{i=1}^n a_i b_i$, a directed set $D$, a weakly null $(x_t)_{t\in \Gamma_{\xi,n}.D}\subset B_X$, and a collection $(x_t^*)_{t\in MAX(\Gamma_{\xi,n}.D)}\subset K$ such that for every $1\leqslant i\leqslant n$ and each $\Lambda_{\xi,n,i}.D\ni s\leqslant t\in MAX(\Gamma_{\xi,n}.D)$, $\text{Re\ }x^*_t(x_s)\geqslant b_i$. Now let $S=\{i\leqslant n: b_i\geqslant \vartheta\}$ and let $T=\{1, \ldots, n\}\setminus S$.  Then \begin{align*} \theta n^{1/q} & \leqslant R\sum_{i\in S} a_i + \vartheta\sum_{i\in T} a_i \\ & \leqslant R|S| + \|(a_i)_{i\in T}\|_{\ell_p^{|T|}} \vartheta |T|^{1/q} \\ & \leqslant R|S|+ \vartheta n^{1/q}.\end{align*} From this it follows that $|S|\geqslant (\theta-\vartheta)n^{1/q}/R\geqslant N_1$.

There exists a map $d:\Gamma_{\xi,|S|}.D\to \Gamma_{\xi,n}.D$ such that, if $S=\{s_1<\ldots <s_{|S|}\}$, $d(\Lambda_{\xi,|S|,i}.D)\subset \Lambda_{\xi, n, s_i}.D$, and such that $(x_{d(s)})_{s\in \Gamma_{\xi, |S|}.D}$ is weakly null.   Let $e:MAX(\Gamma_{\xi, |S|}.D)\to MAX(\Gamma_{\xi, n}.D)$ be any extension of $d$.  Then for any $(s,t)\in \Pi(\Gamma_{\xi,|S|}.D)$, $\text{Re\ }x^*_{e(t)}(x_{d(t)}) \geqslant \theta$.    The collection $(x_{d(t)})_{t\in \Gamma_{\xi, |S|}.D}$ yields that $\theta_{\xi, |S|}(K)\geqslant \theta$, a contradiction.

$(ii)$ By Proposition \ref{too easy} and Theorem \ref{big theorem 2}, $Sz(K)\leqslant \omega^{\xi+1}$ if and only if $\lim_n \theta_{\xi, n}(K)=\inf_n \theta_{\xi,n}(K)=\theta_{\xi+1, 1}(K)=0$.

The equivalence of the last two conditions follows from $(i)$.

\end{proof}

\begin{rem}\upshape From Corollary \ref{p}, it follows that if $\alpha_{\xi,p}(K)=\alpha<\infty$ for some $1<p\leqslant \infty$,   $Sz(K)\leqslant \omega^{\xi+1}$. Indeed, $\lim_n \alpha_{\xi,p,n}(K)/n^{1/q}\leqslant \lim_n \alpha/n^{1/q}=0$.

\label{snit}

\end{rem}

We will need the following. The first part of the proposition is contained in \cite[Proposition $3.5$]{C3}, although not stated in precisely this way. The second part of the following proposition follows immediately from the first.  

\begin{proposition} If $X$ is a Banach space, $K\subset X^*$ is weak$^*$-compact, $b_1, \ldots, b_n$ are non-negative scalars, and if there exists a directed set $D$, a weakly null collection $(x_t)_{t\in \Gamma_{\xi,n}.D}\subset B_X$, and $(x^*_t)_{t\in MAX(\Gamma_{\xi,n}.D)}\subset K$ such that for every $1\leqslant i\leqslant n$ and each $\Lambda_{\xi,n,i}.D\ni s\leqslant t\in MAX(\Gamma_{\xi,n}.D)$, $\text{\emph{Re\ }}x^*_t(x_s)\geqslant b_i$.  Then for any $0<\psi<1$, $$s^{\omega^\xi}_{\psi b_1}\ldots s^{\omega^\xi}_{\psi b_n}(K)\neq \varnothing.$$

In particular, if $\ee>0$, $m\in \nn$ are such that $Sz(K, \ee) \leqslant \omega^\xi m$ and $(x_t)_{t\in \Gamma_{\xi,n}.D}\subset B_X$, $(x^*_t)_{t\in MAX(\Gamma_{\xi,n}.D)}\subset K$, and $b_1, \ldots, b_n$ are as above, $$|\{i\leqslant n: b_i>\ee\}|\leqslant m.$$

\label{downtown abbey}

\end{proposition}

\begin{theorem} Let $X$ be a Banach space, $\xi$ an ordinal, and let $K\subset X^*$ be weak$^*$-compact. \begin{enumerate}[(i)]\item If $q=\max\{\textbf{\emph{p}}_\xi(K), 1\}$ and $1/p+1/q=1$, $$p=\sup\{r\in [1,\infty]:  \alpha_{\xi,r}(K)<\infty\}.$$   \item If $1\leqslant q<\infty$ and $1/p+1/q=1$, $K$ has $\xi$-$q$-summable Szlenk index if and only if $\alpha_{\xi, p}(K)<\infty$. \end{enumerate}

\label{curse}
\end{theorem}

\begin{proof}$(i)$ Let $\beta= \sup\{r\in [1, \infty]:  \alpha_{\xi, r}(K)<\infty\}$ and let $1/\alpha+1/\beta=1$.   The fact that $\beta\leqslant p$  follows from the proof of the main theorem of \cite{C3}.  For the reverse inequality, if $p=1$, $p\leqslant \beta$ and we have $\beta=p$.  So assume $1<p$.  Fix $1\leqslant r<p$ and let $2s=r+p$. In the proof, $r'$, $s'$ denote the conjugate exponents to $r,s$, respectively.  Let $C=\sum_{l=1}^\infty (2^{\frac{1}{r'}-\frac{1}{s'}})^l<\infty$.  Fix $R\geqslant 0$ such that $K\subset RB_{X^*}$.  It follows easily from the definition of $\textbf{p}_\xi(K)$ that there exists $\ee_1\in (0,1)$ such that for every $0<\ee\leqslant \ee_1$, $Sz_\xi(K, \ee)\leqslant 1/\ee^{s'}$. Let $m=1/\ee_1^{s'}$ and $C_1= Rm^{1/r'}+2^{1/s'} m^{1/r'}C$.     Now fix $n\in \nn$, non-negative scalars $(a_i)_{i=1}^n$, $\alpha<\alpha_{\xi,p,n}(K)$, $(x_t)_{t\in \Gamma_{\xi,n}.D}\subset B_X$, functionals $(x^*_t)_{t\in MAX(\Gamma_{\xi,n}.D)}\subset K$, and $b_1, \ldots, b_n\subset [0,R]$ such that $\sum_{i=1}^n a_i b_i>\alpha$ and for each $1\leqslant i\leqslant n$ and each $\Lambda_{\xi,n,i}.D\ni s\leqslant \in MAX(\Gamma_{\xi,n}.D)$, $\text{Re\ }x^*_t(x_s)\geqslant b_i$. For each $l\in \nn$, let $$S_l=\{i\leqslant n: b_i\in (\ee_1 2^{-\frac{l}{s'}}, \ee_1 2^{\frac{1-l}{s'}}]\}$$ and let $$S_0=\{i\leqslant n: b_i\in (\ee_1, R]\}.$$   

 By Proposition \ref{downtown abbey}, since $Sz_\xi(K, \ee_1 2^{-l/s'}) \leqslant 1/(\ee_1 2^{-l/s'})^{s'}=m 2^l$, it follows that $|S_l|\leqslant m 2^l$. Since $Sz(K, \ee_1)\leqslant m$, $|S_0|\leqslant m$.  Now let $$y^*_0= \sum_{j\in S_0} e^*_j$$ and for each $l\in \nn$, let $$y^*_l=\sum_{j\in S_l} 2^{-l/r'} e^*_j.$$    Then $\|y^*_0\|_{\ell_{r'}} = |S_0|^{1/r'} \leqslant m^{1/r'}$ and for each $l\in \nn$, $$\|y^*_l\|_{\ell_{r'}} = 2^{-l/r'} |S_l|^{1/r'} \leqslant 2^{-1/r'} m^{1/r'} 2^{l/r'} = m^{1/r'}.$$   Let $$y^*=Ry^*_0+ 2^{1/s'}\sum_{l\in \nn} (2^{\frac{1}{r'}-\frac{1}{s'}})^l y^*_l$$ and note that $$\|y^*\|_{\ell_{r'}} \leqslant Rm^{1/r'}+ 2^{1/s'}m^{1/r'}C=C_1.$$  Furthermore, \begin{align*} \alpha & < \sum_{i=1}^n a_i b_i  \leqslant \sum_{i\in S_0}a_i b_i+ \sum_{l=1}^\infty \sum_{i\in S_l} a_i b_i \leqslant y^*(\sum_{i=1}^n a_ie_i) \\ & \leqslant C_1\|(a_i)_{i=1}^n\|_{\ell_r}. \end{align*}  Therefore $\alpha_{\xi, r}(K)\leqslant C_1$.

$(ii)$ Assume that $K$ fails to have $\xi$-$q$-summable Szlenk index.  Then for any $M>0$ there exist $\ee_1, \ldots, \ee_n>0$ such that $s^{\omega^\xi}_{\ee_1}\ldots s^{\omega^\xi}_{\ee_n}(K)\neq \varnothing$ and $\sum_{i=1}^n (\ee_i/5)^q>M^q$.   Then arguing as in \cite[Lemma $3.4$]{CAlt}, we may fix a weak neighborhood basis $D$ at $0$ in $X$ and find $(x_t)_{t\in\Gamma_{\xi,n}.D}\subset B_X$, $(x^*_t)_{t\in MAX(\Gamma_{\xi,n}.D)}\subset K$ such that for each $1\leqslant i\leqslant n$ and each $\Lambda_{\xi,n,i}.D\ni s\leqslant t\in MAX(\Gamma_{\xi,n}.D)$, $\text{Re\ }x^*_t(x_s)\geqslant \ee_i/5$. Now fix non-negative numbers $a_1, \ldots, a_n$ such that $\|(a_i)_{i=1}^n\|_{\ell_p^n}=1$ and $\sum_{i=1}^n a_i \ee_i/5 = \bigl(\sum_{i=1}^n (\ee_i/5)^q\bigr)^{1/q}>M$. From this it follows that $$\alpha_{\xi,p, n}(K)\geqslant \inf_{t\in MAX(\Gamma_{\xi,n}.D)} \|\sum_{i=1}^n a_i z^t_i\| \geqslant \sum_{i=1}^n a_i\ee_i/5>M.$$  We therefore deduce that $\alpha_{\xi, p}(K)=\infty$.

Now if $\alpha<\alpha_{\xi,p, n}(K)$, there exist a sequence of positive scalars $(a_i)_{i=1}^n\in B_{\ell_p^n}$ and a weakly null collection $(x_t)_{t\in \Gamma_{\xi,n}.D}\subset B_X$ such that $$\alpha<\inf_{t\in MAX(\Gamma_{\xi,n}.D)}r_K(\sum_{i=1}^n a_i z^t_i).$$    Then there exist non-negative scalars $b_1, \ldots, b_n\in \rr$ such that $\alpha<\sum_{i=1}^n b_i$ and for each $1\leqslant i\leqslant n$ and each $\Lambda_{\xi,n,i}.D\ni s\leqslant t\in MAX(\Gamma_{\xi,n}.D)$, $\text{Re\ }x^*_t(a_i x_s)\geqslant b_i$. Fix $0<\psi<1$ such that $\alpha<\psi\sum_{i=1}^n b_i$.    Then by Proposition \ref{downtown abbey}, $s^{\omega^\xi}_{\psi b_1/a_1}\ldots s^{\omega^\xi}_{\psi b_n/a_n}(K)\neq \varnothing$.  Now \begin{align*} \alpha & < \psi \sum_{i=1}^n b_i = \psi\sum_{i=1}^n a_i (b_i/a_i) \\ & \leqslant \psi \bigl \|(a_i)_{i=1}^n\|_{\ell_p^n}\bigl(\sum_{i=1}^n (b_i/a_i)^q\bigr)^{1/q} \leqslant \psi\bigl(\sum_{i=1}^n (b_i/a_i)^q\bigr)^{1/q}. \end{align*} From this it follows that $K$ does not have $\alpha$-$\xi$-$q$-summable Szlenk index. Therefore if $\alpha_{\xi,p}(K)=\infty$, $K$ does not have $\xi$-$q$-summable Szlenk index.

\end{proof}

\begin{corollary} Let $X$ be a Banach space, $\xi$ an ordinal, $K\subset X^*$ weak$^*$-compact. \begin{enumerate}[(i)]\item For $1\leqslant q<\infty$, $K$ has $\xi$-$q$-summable Szlenk index if and only if $\overline{\text{\emph{abs\ co}}}^{\text{\emph{weak}}^*}(K))$ does. \item $\textbf{\emph{p}}_\xi(\overline{\text{\emph{abs\ co}}}^{\text{\emph{weak}}^*}(K))=\max\{\textbf{\emph{p}}_\xi(K), 1\}$. \end{enumerate}

\label{convex}

\end{corollary}

\begin{proof} Let $1/p+1/q=1$.  By Theorem \ref{big theorem}, $\alpha_{\xi,p, n}(K)=\alpha_{\xi,p,n}(\overline{\text{abs\ co}}^{\text{\emph{weak}}^*}(K))$, whence $\alpha_{\xi,p}(K)<\infty$ if and only  if $\alpha_{\xi,p}(\overline{\text{abs\ co}}^{\text{weak}^*}(K))<\infty$. By Theorem \ref{curse}$(ii)$, the  former condition is equivalent to $K$ having $\xi$-$q$-summable Szlenk index and the latter is equivalent to $\overline{\text{abs\ co}}^{\text{weak}^*}(K)$ having $\xi$-$q$-summable Szlenk index.

We deduce $(ii)$ using $(i)$ together with Theorem \ref{curse}$(i)$.

\end{proof}

\section{Banach ideals}

In this section, we let $\textbf{Ban}$ denote the class of all Banach spaces over $\mathbb{K}$. We let $\mathfrak{L}$ denote the class of all operators between Banach spaces and for $X,Y\in \textbf{Ban}$, we let $\mathfrak{L}(X,Y)$ denote the set of operators from $X$ into $Y$. For $\mathfrak{I}\subset \mathfrak{L}$ and $X,Y\in \textbf{Ban}$, we let $\mathfrak{I}(X,Y)=\mathfrak{I}\cap \mathfrak{L}(X,Y)$. We recall that a class $\mathfrak{I}$ is called an ideal if \begin{enumerate}[(i)]\item for any $W,X,Y,Z\in \textbf{Ban}$, any $C\in \mathfrak{L}(W,X)$, $B\in \mathfrak{I}(X,Y)$, and $A\in \mathfrak{L}(Y,Z)$, $ABC\in \mathfrak{I}$, \item $I_\mathbb{K}\in \mathfrak{I}$, \item for each $X,Y\in \textbf{Ban}$, $\mathfrak{I}(X,Y)$ is a vector subspace of $\mathfrak{L}(X,Y)$. \end{enumerate}

We recall that an ideal $\mathfrak{I}$ is said to be  \begin{enumerate}[(i)]\item \emph{closed} provided that for any $X,Y\in \textbf{Ban}$, $\mathfrak{I}(X,Y)$ is closed in $\mathfrak{L}(X,Y)$ with its norm topology, \item \emph{injective} provided that for any $X,Y,Z\in \textbf{Ban}$, any $A:X\to Y$, and any isomorphic (equivalently, isometric) embedding $j:Y\to Z$ such that $jA\in \mathfrak{I}$, $A\in \mathfrak{I}$, \item \emph{surjective} provided that for any $W,X,Y\in \textbf{Ban}$, any $A:X\to Y$, and any surjection (equivalently, quotient map) $q:W\to X$ such that $Aq\in \mathfrak{I}$, $A\in \mathfrak{I}$. \end{enumerate}

If $\mathfrak{I}$ is an ideal and $\iota$ assigns to each member of $\mathfrak{I}$ a non-negative real value, then we say $\iota$ is an \emph{ideal norm} provided that \begin{enumerate}[(i)]\item for each $X,Y\in \textbf{Ban}$, $\iota$ is a norm on $\mathfrak{I}(X,Y)$, \item for any $W,X,Y,Z\in \textbf{Ban}$ and any $C\in \mathfrak{L}(W,X)$, $B\in \mathfrak{I}(X,Y)$, $A\in \mathfrak{I}(Y,Z)$, $\iota(ABC)\leqslant \|A\|\iota(B)\|C\|$, \item for any $X,Y\in \textbf{Ban}$, any $x\in X$, and any $y\in Y$, $\iota(x\otimes y)=\|x\|\|y\|$. \end{enumerate} 

If $\mathfrak{I}$ is an ideal and $\iota$ is an ideal norm on $\mathfrak{I}$, we say $(\mathfrak{I}, \iota)$ is a \emph{Banach ideal} provided that for every $X,Y\in \textbf{Ban}$, $(\mathfrak{I}(X,Y), \iota)$ is a Banach space.

Fix an ordinal $\xi$,  $1<p\leqslant \infty$, and let $1/p+1/q=1$. For an operator $A:X\to Y$, let $\mathfrak{a}_{\xi,p}(A)=\|A\|+\alpha_{\xi,p}(A)$.   Let $\mathfrak{A}_{\xi,p}$ denote the class of all operators $A$ with $\xi$-$q$-summable Szlenk index.

\begin{theorem} For any ordinal $\xi$ and $1<p\leqslant \infty$, the class $(\mathfrak{A}_{\xi,p}, \mathfrak{a}_{\xi,p})$ is a Banach ideal.

\label{ideal}
\end{theorem}

\begin{proof} Fix Banach spaces $X,Y$. We first show that $\alpha_{\xi,p}$ satisfies the triangle inequality. Fix $A,B:X\to Y$, a directed set $D$,  a weakly null collection $(x_t)_{t\in \Gamma_{\xi,n}.D}\subset B_X$, and a scalar sequence $(a_i)_{i=1}^n\in B_{\ell_p^n}$.  Fix $\alpha_A>\alpha_{\xi,p,n}(A)$ and $\alpha_B>\alpha_{\xi,p,n}(B)$. Then the sets $$\mathcal{E}_A=\{t\in MAX(\Gamma_{\xi,n}.D): \|A\sum_{i=1}^n a_iz_i^t\|\leqslant \alpha_A\}$$ and $$\mathcal{E}_B=\{t\in MAX(\Gamma_{\xi,n}.D): \|B\sum_{i=1}^n a_iz_i^t\|\leqslant\alpha_B\}$$ are eventual. From this it follows that $\mathcal{E}_A\cap \mathcal{E}_B\neq \varnothing$.  Then for any $s\in \mathcal{E}_A\cap \mathcal{E}_B$, $$\inf_{t\in MAX(\Gamma_{\xi,n}.D)} \|(A+B)\sum_{i=1}^n z_i^t\|\leqslant \|(A+B)\sum_{i=1}^n a_i z_i^s\|\leqslant \alpha_A+\alpha_B.$$  From this we deduce the triangle inequality for $\alpha_{\xi,n}$, and therefore for $\mathfrak{a}_{\xi,n}$.  From here it is easy to see that $\mathfrak{a}_{\xi,p}$ is a norm on $\mathfrak{L}(X,Y)$.

Now fix a Banach space $Z$ and fix $A:X\to Y$. For $B:Y\to Z$, a directed set $D$, $n\in \nn$, a scalar sequence $(a_i)_{i=1}^n\in B_{\ell_p^n}$, and a weakly null collection $(x_t)_{t\in \Gamma_{\xi,n}.D}\subset B_X$, $$\inf_{t\in MAX(\Gamma_{\xi,n}.D)} \|BA\sum_{i=1}^n a_i z^t_i\|\leqslant \|B\|\inf_{t\in MAX(\Gamma_{\xi,n})} \|A\sum_{i=1}^n z_i^t\|\leqslant \|B\|\alpha_{\xi,p,n}(A).$$  From this it follows that $\mathfrak{a}_{\xi,p}(BA)\leqslant \|B\|\mathfrak{a}_{\xi,p}(A)$.  For $B:Z\to X$,  a directed set $D$, $n\in \nn$, a scalar sequence $(a_i)_{i=1}^n\in B_{\ell_p^n}$, and a weakly null collection $(x_t)_{t\in \Gamma_{\xi,n}.D}\subset B_X$, \begin{align*} \inf_{t\in MAX(\Gamma_{\xi,n}.D)} \|AB\sum_{i=1}^n z^t_i\| &  \leqslant \inf_{b>\|B\|} b\inf_{t\in MAX(\Gamma_{\xi,n}.D)} \|A\sum_{i=1}^n b^{-1}Bz^t_i\| \leqslant \inf_{b>\|B\|} b\alpha_{\xi,p,n}(A) \\ & = \|B\|\alpha_{\xi,p,n}(A).\end{align*} From this it follows that $\mathfrak{a}_{\xi,p}(AB)\leqslant \|B\|\mathfrak{a}_{\xi,p}$. 

Since $\alpha_{\xi,p}(A)=0$ whenever $A$ is a compact operator, we  deduce that $\mathfrak{A}_{\xi,p}$ is an operator ideal and $\mathfrak{a}_{\xi,p}$ is an ideal norm. 

We last show completeness.  Fix a sequence $(A_k)_{k=1}^\infty$ of $\mathfrak{a}_{\xi,p}$-Cauchy operators from $X$ to $Y$. This is also a norm Cauchy sequence, which must converge in norm to some $A:X\to Y$.  We note that since $\alpha_{\xi,p,n}$ is Lipschitz continuous on $\mathfrak{L}(X,Y)$, \begin{align*} \alpha_{\xi,p}(A) &   =\underset{n}{\sup} \alpha_{\xi,p,n}(A) =  \underset{n}{\ \sup\ } \underset{k}{\ \lim\ } \alpha_{\xi,p,n}(A_k) \leqslant \underset{n}{\ \sup\ }\underset{k}{\ \lim\ }\alpha_{\xi,p}(A_k)= \underset{k}{\ \lim\sup\ }\alpha_{\xi,p}(A_k)<\infty\end{align*} and \begin{align*} \underset{l}{\ \lim\sup\ }\mathfrak{a}_{\xi,p}(A-A_l) & =\underset{l}{\ \lim\sup\ }\alpha_{\xi,p}(A-A_l)  = \underset{l}{\ \lim\sup\ }\underset{n}{\ \sup\ }\underset{k}{\ \lim\sup\ }\alpha_{\xi,p,n}(A_k-A_l) \\ & \leqslant \underset{l}{\ \lim\sup\ }\underset{n}{\ \sup\ }\underset{k}{\ \lim\sup\ }\alpha_{\xi,p}(A_k-A_l) \\ &  = \underset{l}{\ \lim\sup\ }\underset{k}{\ \lim\sup\ }\alpha_{\xi,p}(A_k-A_l)=0 .\end{align*}

\end{proof}

\begin{rem}\upshape We have already shown that $\mathfrak{D}_\xi\subset \mathfrak{A}_{\xi,p}\subset \mathfrak{D}_{\xi+1}$.  It is contained in Section $7$ of \cite{C3} that each of these inclusions is proper.  Furthermore, it is contained in Section $7$ of \cite{C3} that for any ordinals $\xi,\zeta$ and any $1<p,q\leqslant \infty$, $\mathfrak{A}_{\xi, p}\neq \mathfrak{A}_{\zeta,q}$ unless $(\xi, p)=(\zeta, q)$,  if $\xi<\zeta$ or if $\xi=\zeta$ and $p>q$, $\mathfrak{A}_{\xi,p}\subset \mathfrak{A}_{\zeta, q}$, and $\mathfrak{A}_{\xi,q}\subset \mathfrak{T}_{\xi,p}\subset \mathfrak{A}_{\xi,p}$.

\end{rem}

\begin{rem}\upshape It is easy to see that the class of $\xi$-asymptotically uniformly flattenable operators, denoted by $\mathfrak{T}_{\xi,\infty}$ in \cite{AUF}, is contained in $\mathfrak{A}_{\xi,p}$ for every $1<p\leqslant \infty$. It was shown there in \cite[Proposition $6.5$]{AUF} that for any ordinal $\xi$, there exists a Banach space $X$ such that  $\overline{\mathfrak{T}_{\xi,\infty}(X,X)}$ contains an operator which cannot be renormed to be $\xi$-$p$-AUS for any $1<p<\infty$. This example also shows that none of the classes $\mathfrak{A}_{\xi,p}$ is closed.

\end{rem}

Brooker \cite{Br} showed that for any ordinal $\xi$, the class of all operators with Szlenk index not more than $\omega^\xi$ is a closed operator ideal.  An argument similar to that of Theorem \ref{ideal} allows us to provide another proof of this fact which is  dual to Brooker's proof.     

\begin{theorem} For any ordinal $\xi$, the class $\mathfrak{D}_\xi$ of operators with Szlenk index not exceeding $\omega^\xi$ is a closed operator ideal. 

\label{martin}
\end{theorem}

\begin{proof} Arguing as in Theorem \ref{ideal}, one proves that $\theta_{\xi,1}$ defines a seminorm on $\mathfrak{L}(X,Y)$ for each pair $X,Y$ of Banach spaces. Furthermore, $\theta_{\xi,1}(A)=0$ for all compact operators and $\theta_{\xi,1}(ABC)\leqslant \|A\|\theta_{\xi,1}(B)\|C\|$ for any Banach spaces $W,Z$ and any operators $C:W\to X$, $B:X\to Y$, and $A:Y\to Z$. Last, by Theorem \ref{big theorem 2}$(ii)$, $\theta_{\xi,1}$ is $1$-Lipschitz.  From these facts, it follows that the class of all operators $A$ with $\theta_{\xi,1}(A)=0$ is a closed operator ideal. But by Theorem \ref{big theorem 2}$(i)$, this is precisely the class of operators with Szlenk index not exceeding $\omega^\xi$.

\end{proof}

\begin{proposition} For any ordinal $\xi$ and any $1\leqslant p\leqslant \infty$, $\mathfrak{A}_{\xi,p}$ and $\mathfrak{D}_\xi$ are injective and surjective.

\label{ell1}

\end{proposition}

We need the following easy piece. 

\begin{proposition} Suppose $q:W\to X$ is a quotient map, $(x_\lambda)_{\lambda\in D_1}\subset B_X$ is weakly null, and $v$ is a weak neighborhood of $0$ in $W$.  Then for any $b>2$, there exist $\lambda\in D_1$ and  $w\in b B_W\cap v$ such that $qw=x_\lambda$. 

\label{tech}
\end{proposition}

\begin{proof}  By replacing $v$ with a subset, we may assume $v$ is convex and symmetric. Fix $b>2$ and $\delta>0$ such that $\delta B_W\subset \frac{1}{2} v$ and $2+3\delta<b$. For each $\lambda$, fix $u_\lambda \in (1+\delta)B_W$ such that $qu_\lambda= x_\lambda$. By passing to a subnet, we may assume that $(u_\lambda)_{\lambda\in D_1}$ is weak$^*$-convergent to some $u^{**}\in (1+\delta)B_{W^{**}}$, and that $u_{\lambda_1}-u_{\lambda_2}\in \frac{1}{2}v$ for all $\lambda_1, \lambda_2\in D_1$.    Since $(x_\lambda)_{\lambda\in D_1}$ is weakly null, there exist a finite set $F\subset D_1$ and $x=\sum_{\lambda\in F}a_\lambda x_\lambda \in \text{co}(x_\lambda:\lambda\in F)$ such that $\|x\|<\delta$.   We may fix $u\in \delta B_W\subset \frac{1}{2}v$ such that $qu=x$.    Then let $w=u_{\lambda_1}-\sum_{\lambda\in F} a_\lambda u_\lambda + u$.   Then $$\|w\|\leqslant 1+\delta +\sum_{\lambda\in F} a_\lambda(1+\delta)+\delta=2+3\delta <b,$$ $$qw=x_{\lambda_1}-x+x=x_{\lambda_1},$$ and $$w=\sum_{\lambda\in F} a_\lambda( u_{\lambda_1}-u_\lambda)+w\in \frac{1}{2}v+\frac{1}{2}v=v.$$

\end{proof}

\begin{proof}[Proof of Proposition \ref{ell1}] Injectivity is easy, since  for any operator $A:X\to Y$ and any isometric isomorphism $j:Y\to Z$, clearly $\alpha_{\xi,p,n}(jA)=\alpha_{\xi, p,n}(A)$ and $\theta_{\xi,n}(jA)=\theta_{\xi,n}(A)$ for all ordinals $\xi$, all $1\leqslant p\leqslant \infty$, and all $n\in\nn$.

For surjectivity, we will show that if $A:X\to Y$ is an operator and $q:W\to X$ is a quotient map, then $\alpha_{\xi,p, n}(A)\leqslant 2 \alpha_{\xi, p,n}(Aq)$, and a similar argument will yield that $\theta_{\xi,n}(A)\leqslant 2 \theta_{\xi, n}(Aq)$ for any ordinal $\xi$, any $1\leqslant p\leqslant \infty$, and any $n\in \nn$. Fix $\alpha<\alpha_{\xi, p,n}(A)$, a directed set $D$, $(b_i)_{i=1}^n\in S_{\ell_p^n}$, and a weakly null collection $(x_t)_{t\in \Gamma_{\xi,n}.D}\subset B_X$ such that $\inf_{t\in MAX(\Gamma_{\xi,n}.D)} \|A\sum_{i=1}^n b_i z^t_i\|>\alpha$.  Fix $b>2$ such that $\inf_{t\in MAX(\Gamma_{\xi,n}.D)} \|A\sum_{i=1}^n b_i b^{-1}z^t_i\|>\alpha/2$.  Let $D_1$ be any weak neighborhood basis at $0$ in $W$.  We may now recursively apply Proposition \ref{tech} to deduce the existence of some collection $(w_t)_{t\in \Gamma_{\xi,n}.D_1}\subset b B_W$ and a length-preserving monotone map $\phi:\Gamma_{\xi,n}.D_1\to \Gamma_{\xi,n}.D$ such that \begin{enumerate}[(i)]\item if $\phi((\zeta_i, v_i)_{i=1}^k)=(\eta_i, u_i)_{i=1}^k$, then $\zeta_i=\eta_i$ for all $1\leqslant i\leqslant k$, \item $qw_t= x_{\phi(t)}$ for all $t\in \Gamma_{\xi, n}.D_1$. \end{enumerate} Indeed, if $t=s\smallfrown (\zeta, v)\in \Gamma_{\xi,n}.D_1$ and either $s=\varnothing$ or $\phi(s)$ has been defined to have property (i), we apply Proposition \ref{tech} to the weakly null net $(x_{\phi(s)\smallfrown (\zeta,u)})_{u\in D}$ to deduce the existence of some $u_1\in D$ and $w_t\in b B_W\cap v$ such that $qw_t= x_{\phi(s)\smallfrown (\zeta, u_1)}$. We then let $\phi(t)=\phi(s)\smallfrown (\zeta, u_1)$.   

Note that for any $t\in MAX(\Gamma_{\xi,n}.D_1)$ and any $1\leqslant i\leqslant |t|$, $\phi(t|_i)=\phi(t)|_i$ and $\mathbb{P}_{\xi,n}(\phi(t|_i))=\mathbb{P}_{\xi,n}(t|_i)$, whence $$q\sum_{i=1}^n b_i \sum_{t\geqslant s\in \Lambda_{\xi,n,i}.D_1} \mathbb{P}_{\xi,n}(s) b^{-1} w_s= \sum_{i=1}^n b_i \sum_{\phi(t)\geqslant s\in \Lambda_{\xi,n,i}.D} \mathbb{P}_{\xi,n}(s) b^{-1} x_s$$ and \begin{align*} \inf_{t\in MAX(\Gamma_{\xi,n}.D_1)} & \|Aq \sum_{i=1}^n b_i \sum_{t\geqslant s\in \Lambda_{\xi,n,i}.D_1} \mathbb{P}_{\xi,n}(s) b^{-1} w_s\| \\ & = \inf_{t\in MAX(\Gamma_{\xi,n}.D)} \|A\sum_{i=1}^n b_i \sum_{\phi(t)\geqslant s\in \Lambda_{\xi,n,i}.D} \mathbb{P}_{\xi,n}(s) b^{-1} x_s\|>\alpha/2 .\end{align*}     Since $(b^{-1}w_t)_{t\in \Gamma_{\xi,n}.D_1}$ is weakly null and $\alpha<\alpha_{\xi, p,n}(A)$ was arbitrary, we are done.

\end{proof}

\begin{rem}\upshape Note that if $u$ is any weak neighborhood of $0$ in $\ell_1$ and if $m\in \nn$, there exist $m<i<j$, $i,j\in\nn$ such that $\frac{1}{2}(e_i-e_j)\in u$.    From this we can easily deduce that $\theta_{\xi,n}(\ell_1)= 1$ for all ordinals $\xi$ and all $n\in \nn$, and $\alpha_{\xi,p,n}(\ell_1)= n^{1-1/p}$ for all ordinals $\xi$, all $1\leqslant p\leqslant \infty$, and all $n\in \nn$.  Indeed, for any weak neighborhood basis $D$ at $0$ in $\ell_1$, we may recursively define for $t\in \Gamma_{\xi,n}.D$ a vector $x_t= \frac{1}{2}(e_{i_t}-e_{j_t})$ such that for all $t\in \Gamma_{\xi,n}.D$, $i_t<j_t$, if $s,t\in \Gamma_{\xi,n}.D$ and $s<t$, $j_s<i_t$, and if $t=(\zeta_i, u_i)_{i=1}^n$, $x_t\in u_n$.  Then $(x_t)_{t\in \Gamma_{\xi,n}}\subset B_{\ell_1}$ is weakly null and $\|\sum_{i=1}^n a_i z^t_i\|_{\ell_1}= \sum_{i=1}^n |a_i|$ for all scalar sequences $(a_i)_{i=1}^n$.

From this, one can also deduce that if $A:X\to Y$ is an operator,  $Z$ is a subspace of $X$ which isomorphic to $\ell_1$, and if $A|_Z$ is an isomorphic embedding, then $\inf_{\xi\in \textbf{Ord}} \theta_{\xi, 1}(A)>0$. Indeed, we first fix isomorphisms $\alpha:\ell_1\to Z$, $\beta:A(Z)\to \ell_1$ such that $I_{\ell_1}= \beta A \alpha$, so that $\theta_{\xi, 1}(A|_Z:Z\to A(Y)) \geqslant \theta_{\xi,1}(I_{\ell_1})/\|\alpha\|\|\beta\| = 1/\|\alpha\|\|\beta\|$ by the ideal property. Now if $j:A(Z)\to Y$ is the inclusion, it is clear that $$\theta_{\xi,1}(A:X\to Y)\geqslant \theta_{\xi,1}(A|_Z:Z\to Y)= \theta_{\xi,1}(A|_Z:Z\to A(Z)).$$

\label{nonasp}

\end{rem}

\section{Applications of Corollary \ref{convex}}

In this section, we prove an embedding result together with a result concerning injective tensor products.

\begin{corollary} Fix a countable ordinal $\xi$ and $1\leqslant q<\infty$. Let $X$ be a separable Banach space. \begin{enumerate}[(i)]\item If $\textbf{\emph{p}}_\xi(X) \leqslant q$, then there exists a Banach space $W$ with a shrinking basis such that $\textbf{\emph{p}}_\xi(W)\leqslant q$ and $X$ embeds isomorphically into $W$. \item If $X$ has $\xi$-$q$-summable Szlenk index, then there exists a Banach space $W$ with shrinking basis having $\xi$-$q$-summable Szlenk index such that $X$ embeds isomorphically into $W$. \end{enumerate}

\end{corollary}

\begin{proof} Schlumprecht \cite{S} proved that if $X^*$ is separable, after renorming $X$, there exist a weak$^*$-compact set $B^*\subset B_{X^*}$, a Banach space $Z$ with shrinking FDD which contains a subspace isometric to $X$,   a subset $\mathbb{B}^*$ of $B_{Z^*}$, and a map $I^*:\mathbb{B}^*\to B^*$ such that \begin{enumerate}[(i)]\item $\overline{\text{abs\ co}}^{\text{weak}^*}(\mathbb{B}^*)= B_{Z^*}$, \item for any $n\in \nn$ and any $\ee_1, \ldots, \ee_n>0$, $I^*(s^{\omega^\xi}_{\ee_1}\ldots s^{\omega^\xi}_{\ee_n} (\mathbb{B}^*))\subset s^{\omega^\xi}_{\ee_1/5}\ldots s^{\omega^\xi}_{\ee_n/5}(B^*)$. \end{enumerate} The construction of the space $Z$ is the main construction of \cite{S}, while the estimate in $(ii)$ follows from \cite[Lemma $5.5$]{S}. Then for a countable $\xi$ and $1\leqslant q<\infty$, either property $\textbf{p}_\xi(\cdot)\leqslant q$ or having $\xi$-$q$-summable Szlenk index passes from $B_{X^*}$ to $B^*$, since $B^*\subset B_{X^*}$, from $B^*$ to $\mathbb{B}^*$ by $(ii)$ above and Theorem \ref{curse}, and from $\mathbb{B}^*$ to $B_{Z^*}$ by Corollary \ref{convex}.

Now if $(F_n)_{n=1}^\infty$ is the FDD of $Z$, by a technique of Pe\l cz\'{n}ski \cite{Pelc},  for each $n\in \nn$ we may fix a finite dimensional space $H_n$ such that $F_n$ is $2$-complemented in $E_n=F_n\oplus H_n$ and $E_n$ has a basis with basis constant not exceeding $2$.  Then if $H=(\oplus_{n=1}^\infty H_n)_{c_0}$ and if $W=Z\oplus H$, $W$ has a shrinking basis and has $\textbf{p}_\xi(W)\leqslant q$  (resp. $\xi$-summable Szlenk index) if $Z$ has this property.  Here we are using Theorem \ref{curse} together with Theorem \ref{ideal} and the fact that $\textbf{p}_\xi(H)\leqslant \textbf{p}_0(H)=1\leqslant q$ and $\alpha_{\xi, p}(H) \leqslant \alpha_{0,\infty}(H)\leqslant 1$.

\end{proof}

In \cite{DK}, the question was posed as to whether having summable Szlenk index passes to injective tensor products of Banach spaces. This question was answered in the affirmative in \cite{summable}. In Corollary \ref{tensor},  we generalize that result to $\xi$-$q$-summable Szlenk index.

Let us recall that the injective tensor product is the closed span  in $\mathfrak{L}(Y^*, X)$ of the operators $x\otimes y:Y^*\to X$, where $x\otimes y(y^*)= y^*(y)x$. For $i=0,1$, if $A_i:X_i\to Y_i$ is an operator, we may define the operator $A_0\otimes A_1:X_0\hat{\otimes}_\ee X_1\to Y_0\hat{\otimes}_\ee Y_1$. This operator is given by $A_0\otimes A_1 (u)= A_0 u A_1^*:Y^*_1\to Y_0$.  Given subsets $K_0\subset X^*_0$, $K_1\subset X^*_1$, we let $$[K_0, K_1]=\{x^*_0\otimes x^*_1: x^*_0\in K_0, x^*_1\in K_1\}\subset (X_0\hat{\otimes}_\ee X_1)^*.$$ We observe that the map $$X_0^*\otimes X^*_1\supset K_0\times K_1\ni (x^*_0, x^*_1)\mapsto x^*_0\otimes x^*_1\in [K_0, K_1]$$ is weak$^*$-weak$^*$ continuous if $K_0, K_1$ are bounded.     We recall that any non-zero ordinal $\xi$ can be uniquely written as $\xi=\omega^{\xi_1}+\ldots +\omega^{\xi_n}$ for some $\xi_1\geqslant \ldots \geqslant \xi_n$. 

\begin{lemma} Suppose that for $i=0,1$, $A_i:X_i\to Y_i$ is a non-zero operator. Let $R=\max\{\|A_0\|, \|A_1\|\}$.  For  any $\xi>0$,  any $\ee>0$, any finite set $J$, and any weak$^*$-compact sets $K_{0,j}\subset X^*_0$, $K_{1,j}\subset X_1^*$, $j\in J$, $$s_\ee^\xi\Bigl(\bigcup_{j\in J}[K_{0,j}, K_{1,j}]\Bigr)  \subset \bigcup_{j\in J, (k_i)_{i=1}^n\in \{0,1\}^n} [s^{\omega^{\xi_1}k_1+\ldots+ \omega^{\xi_n}k_n}_{\ee/4R}(K_{0,j}), s^{\omega^{\xi_1}(1-k_1)+\ldots +\omega^{\xi_n}(1-k_n)}_{\ee/4R}(K_{1,j})],$$ where $\xi=\omega^{\xi_1}+\ldots+\omega^{\xi_n}$, $\xi_1\geqslant \ldots \geqslant \xi_n$.

\label{tek9}
\end{lemma}

\begin{proof} To obtain a contradiction, assume the lemma does not hold, and let $\xi\geqslant 1$ be the minimum ordinal for which the statement fails. Write $\xi=\omega^{\xi_1}+\ldots +\omega^{\xi_n}$, $\xi_1\geqslant \ldots \geqslant \xi_n$.  We consider three cases. 

Case $1$: $\xi_n=0$.  Then $\omega^{\xi_n}=1$ and $\xi=\zeta+1$, where $\zeta=\omega^{\xi_1}+\ldots +\omega^{\xi_{n-1}}$ if $n>1$ and $\zeta=0$ if $n=1$.  Then by the inductive hypothesis applied to $\zeta$, \begin{align*} s_\ee^\xi\Bigl(&\bigcup_{j\in J}[K_{0,j}, K_{1,j}]\Bigr)    = s_\ee\Bigl(s^\zeta_\ee\Bigl(\bigcup_{j\in J}[K_{0,j}, K_{1,j}]\Bigr)\Bigr) \\ &  \subset s_\ee\Bigl(\bigcup_{j\in J, (k_i)_{i=1}^{n-1}\in \{0,1\}^{n-1}} [s^{\omega^{\xi_1}k_1+\ldots +\omega^{\xi_{n-1}}k_{n-1}}_{\ee/4R}(K_{0,j}), s^{\omega^{\xi_1}(1-k_1)+\ldots +\omega^{\xi_{n-1}}(1-k_{n-1})}_{\ee/4R}(K_{1,j})]\Bigr).\end{align*}  By \cite[Proposition $5.1$]{summable}, \begin{align*} s_\ee\Bigl(& \bigcup_{j\in J, (k_i)_{i=1}^{n-1}\in \{0,1\}^{n-1}} [s^{\omega^{\xi_1}k_1+\ldots \omega^{\xi_{n-1}}k_{n-1}}_{\ee/4R}(K_{0,j}), s^{\omega^{\xi_1}(1-k_1)+\ldots +\omega^{\xi_{n-1}}(1-k_{n-1})}_{\ee/4R}(K_{1,j})]\Bigr) \\ &  \subset \bigcup_{j\in J, (k_i)_{i=1}^n\in \{0,1\}^n} [s_\ee^{k_n}(s^{\omega^{\xi_1}k_1+\ldots \omega^{\xi_{n-1}}k_{n-1}}_{\ee/4R}(K_{0,j})),s^{1-k_n}( s^{\omega^{\xi_1}(1-k_1)+\ldots +\omega^{\xi_{n-1}}(1-k_{n-1})}_{\ee/4R}(K_{1,j})]) \\ & = \bigcup_{j\in J, (k_i)_{i=1}^n\in \{0,1\}^n} [s^{\omega^{\xi_1}k_1+\ldots \omega^{\xi_n}k_n}_{\ee/4R}(K_{0,j}), s^{\omega^{\xi_1}(1-k_1)+\ldots +\omega^{\xi_n}(1-k_n)}_{\ee/4R}(K_{1,j})],\end{align*}  a contradiction. Here we are using the convention that if $n=1$, the first union is taken only over $j\in J$ and $$s_{\ee/4R}^{\omega^{\xi_1}k_1 + \ldots +\omega^{\xi_{n-1}}k_{n-1}}(K_{0,j})=K_{0,j}$$ and $$s_{\ee/4R}^{\omega^{\xi_1}(1-k_1)+ \ldots +\omega^{\xi_{n-1}}(1-k_{n-1})}(K_{1,j})=K_{1,j}.$$

Case $2$: $\xi_n$ is a limit ordinal.  Fix $$u^*\in s^\xi_\ee\Bigl(\bigcup_{j\in J}[K_{0,j}, K_{1,j}]\Bigr) =\bigcap_{\zeta<\xi_n} s^{\omega^{\xi_1}+\ldots +\omega^{\xi_{n-1}}+\omega^\zeta}\Bigl(\bigcup_{j\in J} [K_{0,j}, K_{1,j}]\Bigr).$$  Then for every $\zeta<\xi_n$, there exist $(k^\zeta_i)_{i=1}^n\in \{0,1\}^n$ and $j_\zeta\in J$ such that $$x^*\in [s_\ee^{\omega^{\xi_1}k_1^\zeta+\ldots + \omega^{\xi_{n-1}}k_{n-1}^\zeta+\omega^\zeta k_n^\zeta}(K_{0, j_\zeta}), s_\ee^{\omega^{\xi_1}(1-k_1^\zeta)+\ldots +\omega^{\xi_{n-1}}(1-k_{n-1}^\zeta)+\omega^\zeta(1-k_n^\zeta)}(K_{1, j_\zeta})].$$ Then there exist a cofinal subset $M$ of $[0, \xi_n)$, $j\in J$, and $(k_i)_{i=1}^n\in \{0, 1\}^n$ such that $k_i=k_i^\zeta$ and $j_\zeta=j$ for all $\zeta\in M$.   Then \begin{align*} u^* & \in \bigcap_{\zeta\in M} [s^{\omega^{\xi_1}k_1+\ldots +\omega^{\xi_{n-1}}k_{n-1}+\omega^\zeta k_n}_\ee(K_{0,j}), s^{\omega^{\xi_1}(1-k_1)+\ldots +\omega^{\xi_{n-1}}(1-k_{n-1})+\omega^\zeta(1-k_n)}_\ee(K_{1,j})] .\end{align*} Now for each $\zeta\in M$,  we may fix $$x^*_{0, \zeta}\in s^{\omega^{\xi_1}k_1+\ldots +\omega^{\xi_{n-1}}k_{n-1}+\omega^\zeta k_n}_\ee(K_{0,j})$$ and $$x^*_{1, \zeta}\in s^{\omega^{\xi_1}(1-k_1)+\ldots +\omega^{\xi_{n-1}}(1-k_{n-1})+\omega^\zeta (1-k_n)}_\ee(K_{1,j})$$ such that $x^*_{0,\zeta}\otimes x^*_{1, \zeta}=u^*$.  Then if $$(x^*_0, x^*_1)\in \bigcap_{\mu<\xi_n} \overline{\{(x^*_{0, \zeta}, x^*_{1, \zeta}): \mu<\zeta\in M\} }^{\text{weak}^*} \subset (X_0\oplus X_1)^*,$$ $$u^*=x^*_0\otimes x^*_1\in [s^{\omega^{\xi_1}k_1+\ldots +\omega^{\xi_n}k_n}_\ee(K_{0,j}), s^{\omega^{\xi_1}(1-k_1)+ \ldots +\omega^{\xi_n}(1-k_n)}_\ee(K_{1,j})].$$  Since this holds for any $u^*\in s^\xi_\ee\Bigl(\bigcup_{j\in J}[K_{0,j}, K_{1,j}]\Bigr)$, we reach a contradiction.

Case $3$: $\xi_n$ is a successor ordinal, say $\xi_n=\zeta+1$.  Fix $$u^*\in s^\xi_\ee\Bigl(\bigcup_{j\in J}[K_{0,j}, K_{1,j}]\Bigr) =\bigcap_{m\in \nn} s^{\omega^{\xi_1}+\ldots +\omega_{\xi_{n-1}}+\omega^\zeta m}\Bigl(\bigcup_{j\in J} [K_{0,j}, K_{1,j}]\Bigr).$$  Then for every $m\in \nn$, there exist $j_m\in J$, $(k^m_i)_{i=1}^n\in \{0,1\}^{n-1}$ and $p_m, q_m\subset \{0, \ldots, m\}$ with $p_m+q_m=m$ such that $$x^*\in [s_\ee^{\omega^{\xi_1}k_1^m+\ldots + \omega^{\xi_{n-1}}k_{n-1}^m+\omega^\zeta p_m}(K_{0, j_\zeta}), s_\ee^{\omega^{\xi_1}(1-k_1^m)+\ldots +\omega^{\xi_{n-1}}(1-k_{n-1}^m)+\omega^\zeta q_m}(K_{1, j_\zeta})].$$ Then there exist a cofinal subset $M$ of $\nn$, $j\in J$, and $(k_i)_{i=1}^{n-1}\in \{0, 1\}^{n-1}$ such that $k_i=k_i^m$ and $j_m=j$ for all $m\in M$. Let $k_n=1$ if $\{p_m: m\in M\}$ is unbounded, and otherwise let $q_m=1$.   Then \begin{align*} u^* & \in \bigcap_{m\in M} [s^{\omega^{\xi_1}k_1+\ldots +\omega^{\xi_{n-1}}k_{n-1}+\omega^\zeta p_m}_\ee(K_{0,j}), s^{\omega^{\xi_1}(1-k_1)+\ldots +\omega^{\xi_{n-1}}(1-k_{n-1})+\omega^\zeta q_m}_\ee(K_{1,j})]. \end{align*} Now for each $m\in M$,  we may fix $$x^*_{0, m}\in s^{\omega^{\xi_1}k_1+\ldots +\omega^{\xi_{n-1}}k_{n-1}+\omega^\zeta p_m}_\ee(K_{0,j})$$ and $$x^*_{1, m}\in s^{\omega^{\xi_1}(1-k_1)+\ldots +\omega^{\xi_{n-1}}(1-k_{n-1})+\omega^\zeta q_m}_\ee(K_{1,j})$$ such that $x^*_{0,m}\otimes x^*_{1, m}=u^*$.  Then if $$(x^*_0, x^*_1)\in \bigcap_{l\in \nn} \overline{\{(x^*_{0, m}, x^*_{1, m}): l<m\in M\} }^{\text{weak}^*} \subset (X_0\oplus X_1)^*,$$ $$u^*=x^*_0\otimes x^*_1\in [s^{\omega^{\xi_1}k_1+\ldots +\omega^{\xi_n}k_n}_\ee(K_{0,j}), s^{\omega^{\xi_1}(1-k_1)+ \ldots +\omega^{\xi_n}(1-k_n)}_\ee(K_{1,j})].$$  Since this holds for any $u^*\in s^\xi_\ee\Bigl(\bigcup_{j\in J}[K_{0,j}, K_{1,j}]\Bigr)$, we reach a contradiction.

\end{proof}

Iterating the previous result immediately yields the following. 

\begin{corollary} Suppose that for $i=0,1$, $A_i:X_i\to Y_i$ is a non-zero operator. Let $R=\max\{\|A_0\|, \|A_1\|\}$.  For  any ordinals $\xi_1, \ldots, \xi_n$,  any $\ee_1, \ldots, \ee_n>0$, any finite set $J$, and any weak$^*$-compact sets $K_{0,j}\subset X^*_0$, $K_{1,j}\subset X_1^*$, $j\in J$, \begin{align*} s^{\omega^{\xi_1}}_{\ee_1}\ldots & s^{\omega^{\xi_n}}_{\ee_n}\Bigl(\bigcup_{j\in J}[K_{0,j}, K_{1,j}]\Bigr)  \\ & \subset \bigcup_{j\in J, (k_i)_{i=1}^n\in \{0,1\}^n} [s^{\omega^{\xi_1}k_1}_{\ee_1/4R}\ldots s^{\omega^{\xi_n}k_n}_{\ee_n/4R}(K_{0,j}), s^{\omega^{\xi_1}(1-k_1)}_{\ee_1/4R} \ldots s^{\omega^{\xi_n}(1-k_n)}_{\ee_n/4R}(K_{1,j})].\end{align*}

\label{tek10}

\end{corollary}

\begin{corollary} Fix an ordinal $\xi$ and $1\leqslant q<\infty$. Let $A_0:X_0\to Y_0$, $A_1:X_1\to Y_1$ be  operators and let $A_0\otimes A_1:X_0\hat{\otimes}_\ee X_1\to Y_0\hat{\otimes}_\ee Y_1$ be the induced operator. Then if $A_0, A_1$ have $\xi$-$q$-summable Szlenk index, so does $A_0\otimes A_1$. 
\label{tensor}
\end{corollary}

\begin{proof} If either operator is the zero operator, the result is trivial. Assume $A_0, A_1\neq 0$ and let $R=\max\{\|A_0\|, \|A_1\|\}$.  Suppose $A_0$ has $M_0$-$\xi$-$q$-summable Szlenk index and $A_1$ has $M_1$-$\xi$-$q$-summable Szlenk index. Let $K=[A^*_0B_{Y^*_0}, A^*_1 B_{Y^*_1}]$.   Suppose $\ee_1, \ldots, \ee_n>0$ are such that $s^{\omega^\xi}_{\ee_1}\ldots s^{\omega^\xi}_{\ee_n}(K)\neq \varnothing$. Then by Corollary \ref{tek10}, there exist $(k_i)_{i=1}^n\in \{0,1\}^n$ such that $s^{\omega^\xi k_1}_{\ee_1/4R}\ldots s^{\omega^\xi k_n}_{\ee_n/4R}(A^*_0B_{Y^*_0})\neq \varnothing$ and $s^{\omega^\xi (1-k_1)}_{\ee_1/4R}\ldots s^{\omega^\xi(1-k_n)}_{\ee_n/4R}(A^*_1B_{Y^*_0})\neq \varnothing$.  Then $$\sum_{i=1}^n k_i (\ee_i/4R)^q \leqslant M_0^q$$ and $$\sum_{i=1}^n (1-k_i)(\ee_i/4R)^q \leqslant M_1^q,$$ whence $$\sum_{i=1}^n \ee_i^q \leqslant 4^qR^q (M_0^q+M_1^q),$$ and $K$ has $4R(M_0^q+M_1^q)^{1/q}$-summable Szlenk index. Since $\overline{\text{co}}^{\text{weak}^*}(K)=(A_0\otimes A_1)^*B_{(Y_0\hat{\otimes}_\ee Y_1)^*}$ by the Hahn-Banach theorem, Corollaray \ref{convex} yields that $A_0\otimes A_1$ has $\xi$-$q$-summable Szlenk index.

\end{proof}

In light of Theorem \ref{curse}, Corollary \ref{tensor} also offers another proof that having $\xi$-Szlenk power type not exceeding $q$ also passes to injective tensor products, which was shown in \cite{DK} in the case $\xi=0$ and in the general case in \cite{AUF}.

\section{Direct sums}

In the final section, we are concerned with the behavior of the Szlenk index  and $\alpha_{\xi,p}$ under $\ell_p$ and $c_0$ direct sums.    More specifically, we will have some set $\Lambda$ and some collection $(A_\lambda:X_\lambda\to Y_\lambda)_{\lambda\in \Lambda}$ of Asplund operators such that $\sup_{\lambda\in \Lambda} \|A_\lambda\|<\infty$.  To avoid incredible inconvenience, for a subset $\Upsilon$ of $\Lambda$, we agree that $\ell_0(\Upsilon)$ shall mean $c_0(\Upsilon)$.   We may then define $E_p=(\oplus_{\lambda\in \Lambda}X_\lambda)_{\ell_p(\Lambda)}$ and $F_p=(\oplus_{\lambda\in \Lambda} Y_\lambda)_{\ell_p(\Lambda)}$ for $p\in \{0\}\cup [1, \infty]$.   We also define for each subset $\Upsilon\subset \Lambda$ the projections $P_\Upsilon:E_p\to E_p$, $Q_\Upsilon:F_p\to F_p$ by $P_\Upsilon(x_\lambda)_{\lambda\in \Lambda}= (1_\Upsilon(\lambda) x_\lambda)_{\lambda\in \Lambda}$ and $Q_\Upsilon(y_\lambda)_{\lambda\in \Lambda} = (1_\Upsilon(\lambda) y_\lambda)_{\lambda\in \Lambda}$.  For $\lambda\in \Lambda$, we let $P_\lambda= P_{\{\lambda\}}$ and $Q_\lambda= Q_{\{\lambda\}}$.  For each $p\in \{0\}\cup [1, \infty]$, we then identify $X_\lambda$ with $P_\lambda(E_p)$ and $Y_\lambda$ with $Q_\lambda(F_p)$.  

 Given $p,r\in \{0\}\cup [1, \infty]$ such that either $1\leqslant p\leqslant r\leqslant \infty$ or $r=0$ and $p<\infty$,  and $\Upsilon\subset \Lambda$, we wish to study the behavior of the Szlenk index and  $\alpha_{\xi,\gamma}$ seminorms of the operators $D^{p,r}_\Upsilon:E_p\to F_r$ given by $D^{p,r}_\Upsilon(x_\lambda)_{\lambda\in \Lambda}= (1_\Upsilon(\lambda) A_\lambda x_\lambda)_{\lambda\in \Lambda}$. Of course, we are interested in the cases in which the operator $D^{p,r}_\Lambda$ is Asplund. We first delineate the cases in which the operator $D^{p,r}_\Lambda$ is Asplund.    \begin{enumerate}\item $p=r=1$ and $(\|A_\lambda\|)_{\lambda\in \Lambda}\in c_0(\Lambda)$.  \item $p=r=\infty$ and $(\|A_\lambda\|)_{\lambda\in \Lambda}\in c_0(\Lambda)$. \item  $r=0$ and $p<\infty$.  \item $1\leqslant p\leqslant r\leqslant \infty$, $1<r$, and $p<\infty$.     \end{enumerate}

We first remark that for any finite subset $\Upsilon$ of $\Lambda$, $$\alpha_{\xi,\gamma, n}(D^{p,r}_\Upsilon)= \alpha_{\xi, \gamma, n}\Bigl(\oplus_{\lambda\in \Upsilon} A_\lambda:(\oplus_{\lambda\in \Upsilon} X_\lambda)_{\ell_p(\Upsilon)}\to (\oplus_{\lambda\in \Upsilon} Y_\lambda)_{\ell_r(\Upsilon)}\Bigr)$$ and $$\theta_{\xi, n}(D^{p,r}_\Upsilon)= \theta_{\xi, n}\Bigl(\oplus_{\lambda\in \Upsilon} A_\lambda:(\oplus_{\lambda\in \Upsilon} X_\lambda)_{\ell_p(\Upsilon)}\to (\oplus_{\lambda\in \Upsilon} Y_\lambda)_{\ell_r(\Upsilon)}\Bigr)$$ for any $1\leqslant \gamma\leqslant \infty$, any $n\in \nn$, and any ordinal $\xi$.  Therefore we may freely identify $D^{p,r}_\Upsilon$ with the operators between finite direct sums.

We isolate the following easy consequence of H\"{o}lder's inequality. 

\begin{fact} Suppose $1\leqslant r< \gamma \leqslant \infty$, and $1/\gamma+1/s=1/r$.    Then for any $n\in \nn$ and any scalar sequence $(\alpha_i)_{i=1}^n$,  $$\|(\alpha_i)_{i=1}^n\|_{\ell^n_s} = \sup \{\|(\alpha_i\beta_i)_{i=1}^n\|_{\ell_r^n}: (\beta_i)_{i=1}^n\in B_{\ell^n_\gamma}\}.$$   

\label{f1}
\end{fact}

\begin{fact} Fix $1\leqslant p,r, \gamma\leqslant \infty$ with $p\leqslant r$.   For any natural numbers $m,n$, any scalar sequence $(\alpha_i)_{i=1}^m$, any $(b_j)_{j=1}^n\in B_{\ell_\gamma^n}$, and any $(a_{ij})_{i=1, j=1}^{mn}\in B_{\ell_\infty^n(\ell_p^m)}$, \begin{displaymath}
   \bigl\|\bigl(\|(\alpha_i b_j a_{ij})_{j=1}^n\|_{\ell_\gamma^n}\bigr)_{i=1}^m\bigr\|_{\ell_r^m} \leqslant  \left\{
     \begin{array}{lr}
       \|(\alpha_i)_{i=1}^m\|_{\ell_\infty^m} & : \gamma\leqslant r\\
       \|(\alpha_i)_{i=1}^m\|_{\ell_s^m} & : \gamma>r \text{\ \emph{and}\ }\frac{1}{\gamma}+\frac{1}{s}=\frac{1}{r}.
     \end{array}
   \right.
\end{displaymath} 

\label{f2}
\end{fact}

\begin{proof} Suppose $\gamma\leqslant r$.  Let $\alpha=\|(\alpha_i)_{i=1}^m\|_{\ell_\infty^m}$.    Then using Minkowski's inequality, \begin{align*} \bigl\|\bigl(\|(\alpha_i b_j a_{ij})_{j=1}^n\|_{\ell_\gamma^n}\bigr)_{i=1}^m\bigr\|_{\ell_r^m} & \leqslant \alpha \bigl\|\bigl(\|( b_j a_{ij})_{j=1}^n\|_{\ell_\gamma^n}\bigr)_{i=1}^m\bigr\|_{\ell_r^m} \\ & \leqslant \alpha \bigl\|\bigl(|b_j|\|(a_{ij})_{i=1}^m\|_{\ell_\gamma^n}\bigr)_{j=1}^n\bigr\|_{\ell_r^m}  \\ & \leqslant \alpha \|(b_j)_{j=1}^n\|_{\ell_\gamma^n} \leqslant \alpha. \end{align*}

Now assume that $\gamma>r$.  Then \begin{align*} \bigl\|\bigl( \|(b_ja_{ij})_{j=1}^n\|_{\ell_\gamma^n}\bigr)_{i=1}^m \bigr\|_{\ell_\gamma^m} & = \bigl\|\bigl(b_j \|(a_{ij})_{i=1}^m\|_{\ell_\gamma^m}\bigr)_{j=1}^n\bigr\|_{\ell_\gamma^n} \\ & \leqslant \bigl\|\bigl(b_j \|(a_{ij})_{i=1}^m\|_{\ell_p^m}\bigr)_{j=1}^n\bigr\|_{\ell_\gamma^n} \\ & \leqslant \|(b_j)_{j=1}^n\|_{\ell_\gamma^n} \leqslant 1. \end{align*} Since $\frac{1}{\gamma}+\frac{1}{s}=\frac{1}{r}$, Fact \ref{f1} yields that $$\bigl\|\bigl(\|(\alpha_i b_j a_{ij})_{j=1}^n\|_{\ell_\gamma^n}\bigr)_{i=1}^m\bigr\|_{\ell_r^m} \leqslant \|(\alpha_i)_{i=1}^m\|_{\ell_s^m}.$$

\end{proof}

\begin{theorem} Suppose $\Upsilon\subset \Lambda$ is a finite set and $R\geqslant 0$ is such that $\|A_\lambda\|\leqslant R$ for all $\lambda\in \Lambda$.   Fix $1\leqslant p,r, \gamma\leqslant \infty$ with $p\leqslant r$ and an ordinal $\xi$.  \begin{enumerate}[(i)]\item Let $1/r+1/s=1$.  If $\lambda\in\Lambda$ is such that $\theta_{\zeta, m}(D^{p,r}_\lambda)\leqslant \ee$, then for any directed set $D$, any $n\in \nn$, and sequence $(b_i)_{i=1}^n$ of non-negative scalars, any weakly null $(x_t)_{t\in \Gamma_{\zeta,n}.D}\subset B_{E_p}$,  any sequence of positive scalars $(a_i)_{i=1}^n\in B_{\ell_p(\Upsilon)}$ such that $\|P_\lambda x_t\|\leqslant a_i$ for all $1\leqslant i\leqslant n$, and any $\ee'>\ee\|(b_i a_i)_{i=1}^n\|_{\ell_1^n}+ R m^{1/s}\|(b_i a_i)_{i=1}^n\|_{\ell_r^n}$, $$\{t\in MAX(\Gamma_{\zeta, n}.D): \|D^{p,r}_\lambda \sum_{i=1}^n b_i z^t_i\|\leqslant \ee'\}$$ is eventual.    \item Let $s$ be such that $1/r+1/s=1$. If $\zeta$ is an ordinal and $m\in \nn$ are such that $\theta_{\zeta, m}(A_\lambda)\leqslant \ee$ for all $\lambda\in \Upsilon$, then for any directed set $D$, any $n\in \nn$,  any weakly null $(x_t)_{t\in \Gamma_{\zeta,m}.D}\subset B_{E_p}$, and any non-negative scalars $(b_i)_{i=1}^n$,  $$\inf_{t\in MAX(\Gamma_{\xi,m}.D)} \|D^{p,r}_\Upsilon \sum_{i=1}^n b_i z^t_i \|\leqslant \ee \|(b_i)_{i=1}^n\|_{\ell_1^n}+ Rm^{1/s}\|(b_i)_{i=1}^n\|_{\ell_r^n}$$ and  and for any $\ee_1> \ee \|(b_i)_{i=1}^n\|_{\ell_1^n}+Rm^{1/s} \|(b_i)_{i=1}^n\|_{\ell_r^n}$, $$\{t\in MAX(\Gamma_{\xi,m}.D): \|D^{p,r}_\Upsilon \sum_{i=1}^m b_i z^t_i\|\leqslant \ee_1 \}$$ is eventual.    \item  \begin{displaymath}
   \alpha_{\xi, \gamma}(D^{p,r}_\Upsilon) =  \left\{
     \begin{array}{lr}
       \|(\alpha_{\xi, \gamma}(A_\lambda))_{\lambda\in \Upsilon}\|_{\ell_\infty(\Upsilon)} & : \gamma\leqslant r\\
       \|(\alpha_{\xi, \gamma}(A_\lambda))_{\lambda\in \Upsilon}\|_{\ell_s(\Upsilon)}  & : \gamma>r \text{\ \emph{and}\ }\frac{1}{\gamma }+\frac{1}{s}=\frac{1}{r}.
     \end{array}
   \right.
\end{displaymath}  \end{enumerate}

\label{30r}

\end{theorem}

Before we prove the theorem, we prove a single instance of a stabilization technique, variants of which will be used several times in the remainder of this work.   

\begin{lemma}  If $\alpha<\alpha_{\xi, \gamma, n}(D^{p,r}_\Upsilon)$, there exist a sequence of  non-negative scalars $(b_i)_{i=1}^n \in S_{\ell_\gamma(\Upsilon)}$, real numbers $c_1, \ldots, c_n$ such that $\sum_{i=1}^n b_i c_i>\alpha$, vectors $( a^i_\lambda)_{\lambda\in \Upsilon}$, $i=1, \ldots n$ having positive coefficients such that  $\|( a^i_\lambda)_{\lambda\in\Upsilon}\|_{\ell_p(\Upsilon)}\leqslant 1$,  a directed set $D$,  a weakly null collection $(x_t)_{t\in \Gamma_{\xi,n}.D}\subset B_{E_p}$, and functionals $(y^*_t)_{t\in MAX(\Gamma_{\xi,n}.D)}\subset B_{F_r^*}$ such that for each $1\leqslant i\leqslant n$, \begin{enumerate}[(i)]\item for each $\lambda\in \Upsilon$ and $t\in \Lambda_{\xi,n,i}.D$, $\|P_\lambda x_t\|\leqslant a^i_\lambda$,   \item for each $\Lambda_{\xi,n,i}.D\ni s\leqslant t\in MAX(\Gamma_{\xi,n}.D)$, $\text{\emph{Re\ }}y^*_t(D^{p,r}_\Upsilon x_s)\geqslant c_i$, \item for any $\ee_\lambda>\alpha_{\xi, \gamma ,n}(A_\lambda)\|(a^i_\lambda b_i)_{i=1}^n\|_{\ell_\gamma^n}$, $$\{t\in MAX(\Gamma_{\xi, \gamma, n}.D): \|D^{p,r}_\lambda \sum_{j=1}^n b_j z^t_j\|\leqslant \ee_\lambda\}$$ is eventual.  \end{enumerate}

\label{kimmy}
\end{lemma}

\begin{proof} We first fix $0<\phi<1$ such that $\alpha<\phi \alpha_{\xi, \gamma, n}(D^{p,r}_\Upsilon)$.  We fix $(x'_t)_{t\in \Gamma_{\xi,n}.D}\subset B_{E_p}$, $(y^{',*}_t)_{t\in MAX(\Gamma_{\xi,n}.D)}\subset B_{F_r^*}$, $(b_i)_{i=1}^n\in S_{\ell_\gamma^n}$, non-negative scalars $(c_i)_{i=1}^n$ such that $\sum_{i=1}^n b_i c_i>\alpha/\phi$ and for any $1\leqslant i\leqslant n$ and $\Lambda_{\xi,n,i}.D\ni s\leqslant t\in MAX(\Gamma_{\xi,n}.D)$, $\text{Re\ }y^{',*}_t(x'_s) \geqslant c_i$.   Now fix a finite subset $F$ of $B_{\ell_p(\Upsilon)}$ such that for every $(a'_\lambda)_{\lambda\in \Upsilon}\in \phi B_{\ell_p(\Upsilon)}$, there exists $(a_\lambda)_{\lambda\in \Upsilon}\in F$ such that $|a'_\lambda|\leqslant a_\lambda$, and such that for each $(a_\lambda)_{\lambda\in \Upsilon}\in F$ and $\lambda\in \Upsilon$, $a_\lambda>0$.       Now define $f:\Gamma_{\xi,n}.D\to F$ by letting $f(t)$ be any member $(a_\lambda)_{\lambda\in \Upsilon}$ of $F$ such that for each $\lambda\in \Lambda$, $\phi \|P_\lambda x_t'\|\leqslant a_\lambda$.   By Lemma \ref{stabilize}, there exist a level map $d:\Gamma_{\xi,n}.D\to \Gamma_{\xi,n}.D$, an extension $e$ of $d$, and for each $1\leqslant i\leqslant n$ some $(a^i_\lambda)_{\lambda\in \Upsilon}\in F$ such that $\phi\|P_\lambda x_{d(t)}'\|\leqslant a^i_\lambda$ for all $t\in \Lambda_{\xi,n,i}.D$ and $\lambda\in\Upsilon$.  Now let $x_t=\phi x'_{d(t)}$ and $y^*_t=y^{',*}_{e(t)}$ for $t\in MAX(\Gamma_{\xi,n}.D)$.     We now note that for any $\ee_\lambda> \alpha_{\xi, \gamma ,n}(A_\lambda)\|(a^i_\lambda b_i)_{i=1}^n\|_{\ell_\gamma^n}$, \begin{align*} \{t\in MAX(\Gamma_{\xi, \gamma, n}.D) & : \|D^{p,r}_\lambda \sum_{j=1}^n b_j z^t_j\|\leqslant \ee_\lambda\} \\ & =\{t\in MAX(\Gamma_{\xi, \gamma, n}.D): \|D^{p,r}_\lambda \sum_{j=1}^n \sum_{t\geqslant s\in \Lambda_{\xi,n,j}.D} b_j a^j_\lambda \mathbb{P}_{\xi,n}(s) \frac{P_\lambda x_s}{a^j_\lambda}\|\leqslant \ee_\lambda\}  \end{align*} must be eventual.

\end{proof}

\begin{proof} [Proof of Theorem \ref{30r}]$(i)$  If it were not true, after relabeling, we can assume there exist $(y^*_t)_{t\in MAX(\Gamma_{\zeta, n}.D)}\subset B_{F^*_r}$, $c_1, \ldots, c_n\in [0,\infty)$ such that $\sum_{i=1}^n b_i c_i > \ee\|(b_i a_i)_{i=1}^n\|_{\ell_1^n}+ R m^{1/s}\|(b_i a_i)_{i=1}^n\|_{\ell_r^n}$ and for each $1\leqslant i\leqslant n$ and $\Lambda_{\zeta, n, i}.D\in s\leqslant t\in MAX(\Gamma_{\zeta,n}.D)$, $\text{Re\ }y^*_t(P_\lambda x_s)\geqslant c_i$.    Now let $T=\{i\leqslant n: c_i/a_i>\ee\}$ and $S=\{1, \ldots, n\}\setminus T$.   We note that $|T|<m$. If it were not so, we could find some $l_1<\ldots<l_m$, $l_i\in T$, and a level map $d:\Gamma_{\zeta, m}.D\to \Gamma_{\zeta, n}.D$ such that $d(\Lambda_{\zeta, m, i}.D)\subset \Lambda_{\zeta, n, l_i}.D$ and such that $(x_{d(t)})_{t\in \Gamma_{\zeta, m}.D}$ is weakly null.   Now if $x'_t=a_{l_i}^{-1}x_{d(t)}$ for $t\in \Lambda_{\xi, m, i}.D$, $(x'_t)_{t\in \Gamma_{\zeta, m}.D}$ is weakly null and if $e$ is any extension of $d$, it $$\ee\geqslant \theta_{\zeta, m}(D^{p,r}_\lambda)>\inf_{t\in MAX(\Gamma_{\zeta, m}.D)} m^{-1}\text{Re\ }y^*_{e(t)}\sum_{i=1}^m \sum_{t\geqslant s\in \Lambda_{\zeta, m,i}.D} \mathbb{P}_{\zeta, m}(s) x_s \geqslant \frac{\min \{c_i/a_i: i\in T\}}{m}>\ee,$$ a contradiction. Note also that for any $1\leqslant i\leqslant n$, if we fix $\Lambda_{\zeta, n,i}.D\ni s\leqslant t\in MAX(\Gamma_{\zeta, n}.D)$, $$c_i/a_i \leqslant \text{Re\ }y^*_t(D^{p,r}_\lambda x_s/a_i) \leqslant \|y^*_t\|\|A_\lambda\|\|P_\lambda x_s\|/a_i \leqslant R.$$

Then \begin{align*} \ee\|(b_i a_i)_{i=1}^n\|_{\ell_1^n}+ R m^{1/s}\|(b_i a_i)_{i=1}^n\|_{\ell_r^n} & < \sum_{i=1}^n b_i c_i = \sum_{i=1}^n a_i b_i \frac{c_i}{a_i} \\ & = \sum_{i\in S} a_ib_i\frac{c_i}{a_i} + \sum_{i\in T} a_ib_i\frac{c_i}{a_i} \\ & \leqslant \ee\sum_{i=1}^n a_ib_i + R\sum_{i\in T} a_ib_i \\ & \leqslant \ee\|(a_ib_i)_{i=1}^n\|_{\ell_1^n}+ R\|(a_ib_i)_{i\in T}\|_{\ell_1(T)} \\ & \leqslant \ee\|(a_ib_i)_{i=1}^n\|_{\ell_1^n}+ Rm^{1/s}\|(a_ib_i)_{i\in T}\|_{\ell_r(T)} \\ & \leqslant \ee\|(a_ib_i)_{i=1}^n\|_{\ell_1^n}+ Rm^{1/s}\|(a_ib_i)_{i=1}^n\|_{\ell_r^n},\end{align*} a contradiction.

 $(ii)$ Fix a sequence $(b_i)_{i=1}^n$ of non-negative scalars and $\ee_1> \ee\|(b_i)_{i=1}^n\|_{\ell_1^n}+ Rm^{1/s}\|(b_i)_{i=1}^n\|_{\ell_r^n}$. To obtain a contradiction, we assume that there exist a directed set $D$ and a weakly null collection $(x_t)_{t\in \Gamma_{\zeta,n}.D}\subset B_{E_p}$ such that $$\{t\in MAX(\Gamma_{\zeta,n}.D): \|D^{p,r}_\Upsilon\sum_{i=1}^n b_i z^t_i\|>\ee_1\}$$ is cofinal, in which case we may assume that there exist $\inf_{t\in MAX(\Gamma_{\zeta,n}.D)} \|D^{p,r}_\Upsilon\sum_{i=1}^n b_i z^t_i\|\geqslant \ee_1$ and $c_1, \ldots, c_n\in [0, R]$, $(y^*_t)_{t\in MAX(\Gamma_{\zeta,n}.D)}\subset B_{F_r^*}$ such that $\ee_1<\sum_{i=1}^n b_i c_i$ and for any $1\leqslant i\leqslant n$ and $\Lambda_{\zeta,n,i}.D\ni s\leqslant t\in MAX(\Gamma_{\zeta,n}.D)$, $\text{Re\ }y^*_t(x_s) = \text{Re\ }y^*_t(P_\Upsilon x_s) \geqslant c_i$. Fix $0<\phi<1$ such that $\ee_1<\sum_{i=1}^n \phi b_i c_i$.  Arguing as in Lemma \ref{kimmy}, we may stabilize, relabel, and assume that there exists $(a_{i\lambda})_{i=1, \lambda\in \Upsilon}^n \in B_{\ell_\infty^n(\ell_p(\Upsilon))}$ such that for each $1\leqslant i\leqslant n$, each $t\in \Lambda_{\zeta, n,i}.D$, and each $\lambda\in \Upsilon$, $\phi\|P_\lambda x_t\|\leqslant a_{i \lambda}$.  Now fix $\ee'>0$ such that $$\ee_1- n\ee' > \ee\|(b_i)_{i=1}^n\|_{\ell_1^n}+ Rm^{1/s}\|(b_i)_{i=1}^n\|_{\ell_r^n}.$$  For each $\lambda\in \Upsilon$, $$\mathcal{E}_\lambda:=\{t\in MAX(\Gamma_{\zeta, n}.D): \|D^{p,r}_\lambda\sum_{i=1}^n b_i \phi z^t_i\| \leqslant \ee'+ \ee \|(b_i a_{ij})_{i=1}^n\|_{\ell_1^n}+ Rm^{1/s}\|(b_i a_{ij})_{i=1}^n\|_{\ell_r^n}\}$$ is eventual by (i).   Then if $t\in \cap_{\lambda\in \Upsilon} \mathcal{E}_\lambda$, by Minkowski's inequality, \begin{align*} \ee_1 & \leqslant \|D^{p,r}_\Upsilon \sum_{i=1}^n b_i z^t_i\| \leqslant n\ee' + \ee\bigl\|\bigl(\|(b_i a_{i\lambda})_{i=1}^n\|_{\ell_1^n}\bigr)_{\lambda\in \Upsilon} \bigr\|_{\ell_r(\Upsilon)} + Rm^{1/s} \bigl\|\bigl(\|(b_i a_{i\lambda})_{i=1}^n\|_{\ell_r^n}\bigr)_{\lambda\in\Upsilon}\bigr\|_{\ell_r(\Upsilon)} \\ & \leqslant n \ee'  + \ee\bigl\|\bigl(b_i \|(a_{i\lambda})_{\lambda\in \Upsilon}\|_{\ell_r(\Upsilon)}\bigr)_{i=1}^n\bigr\|_{\ell_1^n} + Rm^{1/s}\bigl\|\bigl(b_i \|(a_{i\lambda})_{\lambda\in \Upsilon}\bigr)_{i=1}^n\|_{\ell_r(\Upsilon)}\bigr\|_{\ell_r^n} \\ & \leqslant  n \ee'  + \ee\|(b_i )_{i=1}^n\|_{\ell_1^n} + Rm^{1/s}\|(b_i)_{i=1}^n\|_{\ell_r^n}, \end{align*} a contradiction.

$(iii)$ If  $\alpha< \alpha_{\xi,\gamma, n}(D^{p,r}_\Upsilon)$, we may fix $(x_t)_{t\in \Gamma_{\xi,n}.D}\subset B_{E_p}$, $(a^i_\lambda)_{\lambda\in \Upsilon}\in B_{\ell_p^n}$, $i=1, \ldots, n$,  $(x^*_t)_{t\in MAX(\Gamma_{\xi,n}.D)}\subset B_{F_r^*}$ to satisfy the conclusions of Lemma \ref{kimmy}.   Now for each $\lambda\in \Upsilon$, fix $\ee_\lambda> \alpha_{\xi, \gamma ,n}(A_\lambda)\|(a^i_\lambda b_i)_{i=1}^n\|_{\ell_\gamma^n}$ and $t\in MAX(\Gamma_{\xi,n}.D)$ such that for each $\lambda\in \Upsilon$, $\|D^{p,r}_\lambda \sum_{i=1}^n b_i z^t_i\| \leqslant \ee_\lambda$. For each $\lambda\in \Upsilon$, the set $\mathcal{E}_\lambda$ of $t\in MAX(\Gamma_{\xi,n}.D)$ satisfying this inequality is eventual, so $\cap_{\lambda\in \Upsilon} \mathcal{E}_\lambda\neq \varnothing$.    Then \begin{align*} \alpha & <\inf_{s\in MAX(\Gamma_{\xi,n}.D)} \|D^{p,r}_\Upsilon \sum_{i=1}^n b_i z^s_i\| \leqslant  \|D^{p,r}_\Upsilon\sum_{i=1}^n b_i z^t_i\|  = \Biggl\|\Bigl(\|D^{p,r}_\lambda \|\sum_{i=1}^n b_i z^t_i\|\Bigr)_{\lambda\in \Upsilon} \Biggr\|_{\ell_r^\Upsilon} \\ & \leqslant \|(\ee_\lambda)_{\lambda\in \Upsilon}\|_{\ell_r(\Upsilon)}.\end{align*}  Since we may do this for any $\ee_\lambda>\alpha_{\xi, \gamma ,n}(A_\lambda)\|(a^i_\lambda b_i)_{i=1}^n\|_{\ell_\gamma^n}$, we deduce that $$\alpha \leqslant \Bigl\|\Bigl(\bigl\|(\alpha_{\xi, \gamma, n}(A_\lambda) a^i_\lambda b_i)_{i=1}^n\bigr\|_{\ell_\gamma^n} \Bigr)_{\lambda\in \Upsilon}\Bigr\|_{\ell_r(\Upsilon)}.$$  By Fact \ref{f2}, the last quantity cannot exceed $\|(\alpha_{\xi, \gamma, n}(A_\lambda))_{\lambda\in \Upsilon}\|_{\ell_\infty(\Upsilon)}$ if either $r=0$ or $\gamma\leqslant r$, and cannot exceed $\|(\alpha_{\xi, \gamma, n}(A_\lambda))_{\lambda\in \Upsilon}\|_{\ell_s(\Upsilon)}$ if $1\leqslant r<\gamma$ and $1/\gamma+1/s=1/r$. Here we are using the fact that $\ell_0(\Upsilon)=\ell_\infty(\Upsilon)$ since $\Upsilon$ is finite. 

For convenience, let us assume $\Upsilon = \{1, \ldots, m\}$. Fix $0<\phi<\varphi<1$ and $k\in \nn$ such that for each $i\in \Upsilon$, $\alpha_{\xi, \gamma, k}(A_i)\geqslant \varphi \alpha_{\xi, \gamma}(A_i)$.    Let $D$ be a weak neighborhood basis at $0$ in $X$. For each $i\in \Upsilon$, we may fix $(b^i_j)_{j=1}^k\in S_{\ell_\gamma^k}$ and $(x_t^i)_{t\in \Gamma_{\xi,k}.D}\subset B_{E_p}$ such that $x_t^i=P_i x^i_t$ for all $t\in \Gamma_{\xi,k}.D$ and $$\inf_{t\in MAX(\Gamma_{\xi,k}.D)} \|D^{p,r}_i \sum_{j=1}^k \sum_{t\geqslant s\in \Lambda_{\xi,k,j}.D} b_j^i \mathbb{P}_{\xi,k}(s) x^i_s\| \geqslant \phi\alpha_{\xi, \gamma}(A_i).$$    By Fact \ref{f1}, there exists $(c_i)_{i\in \Upsilon}\in S_{\ell_\gamma^m}$ such that $$\|(\alpha_{\xi, \gamma}(A_i) c_i)_{i=1}^m \|_{\ell_r^m}= \|(\alpha_{\xi, \gamma}(A_i))_{i=1}^m\|_{\ell_s^m}.$$  Now for $1\leqslant i\leqslant m$ and $1\leqslant j\leqslant k$, let $d_{(i-1)k+j}=c_ib^i_j$ and note that $\|(d_i)_{i=1}^{mk}\|_{\ell_\gamma^{mk}}=1$. For $1\leqslant i\leqslant m$,  $1\leqslant j\leqslant k$, and  $t\in \Lambda_{\xi, mk, (i-1)k+j}.D$, we may write $t=s_1\smallfrown (\omega^\xi(k(m-i))+s)$, where $s_1\in MAX(\Lambda_{\xi, mk, (i-1)k}.D)$ if $i>1$ and $s_1=\varnothing$ if $i=1$, and where $s\in \Gamma_{\xi, k}.D$.   We then let $x_t=u^i_s$.  Then $(x_t)_{t\in \Gamma_{\xi, mk}.D}\subset B_{E_p}$ is weakly null and \begin{align*} \inf_{t\in MAX(\Gamma_{\xi, mk}.D)} \|D^{p,r}_\Upsilon \sum_{i=1}^{mk} d_i z^t_i\| & \geqslant \phi\|(c_i\alpha_{\xi, \gamma}(A_i))_{i=1}^m\|_{\ell_r^m} = \phi\|(\alpha_{\xi, \gamma}(A_i))_{i=1}^m\|_{\ell^s_m}. \end{align*}    This yields that $\alpha_{\xi, \gamma}(D^{p,r}_\Upsilon)\geqslant \phi\|(\alpha_{\xi, \gamma}(A_i))_{i=1}^m\|_{\ell^s_m}$.  Since this holds for any $0<\phi<1$, we deduce that $\alpha_{\xi, \gamma}(D^{p,r}_\Upsilon)\geqslant \|(\alpha_{\xi, \gamma}(A_i))_{i=1}^m\|_{\ell_s^m}$.

\end{proof}

In the following, if for each $\lambda\in \Lambda$, $\alpha_\lambda\in [0, \infty]$, we let $\|(\alpha_\lambda)_{\lambda\in \Lambda}\|_{\ell_s(\Lambda)}=\infty$ if either $\alpha_\lambda=\infty$ for some $\lambda\in \Lambda$ or if $\alpha_\lambda\in [0, \infty)$ for all $\lambda$ and $(\alpha_\lambda)_{\lambda\in \Lambda}\notin \ell_s(\Lambda)$.

\begin{corollary} \begin{enumerate}[(i)]\item $D^{1,1}_\Lambda$ is Asplund if and only if $(A_\lambda)_{\lambda\in \Lambda}\in c_0(\Lambda)$, in which case $$Sz(D^{1,1}_\Lambda)= \sup_{\lambda\in\Lambda} Sz(A_\lambda)$$ and for any $1\leqslant \gamma\leqslant \infty$, $\alpha_{\xi, \gamma}(D^{1,1}_\Lambda)= \|(\alpha_{\xi, \gamma}(A_\lambda))_{\lambda\in \Lambda}\|_{\ell_s(\Lambda)}$, where $1/\gamma+1/s=1$.

\item $D^{\infty, \infty}_\Lambda$ is Asplund if and only if $(A_\lambda)_{\lambda\in \Lambda}\in c_0(\Lambda)$, in which case $$Sz(D^{\infty, \infty}_\Lambda)=\sup_{\lambda\in \Lambda} Sz(A_\lambda)$$ and for any $1\leqslant \gamma\leqslant \infty$, $\alpha_{\xi, \gamma}(D^{\infty, \infty}_\Lambda) = \|(\alpha_{\xi, \gamma}(A_\lambda))_{\lambda\in \Lambda}\|_{\ell_\infty(\Lambda)}$. 

\end{enumerate}

\label{tgs}

\end{corollary}

\begin{proof} $(i)$ If $(\|A_\lambda\|)_{\lambda\in \Lambda}\in \ell_\infty(\Lambda)\setminus c_0(\Lambda)$, then $D^{1,1}_\Lambda$ preserves an isomorph of $\ell_1$ and $\sup_{\xi\in \textbf{Ord}}\theta_{\xi,1}(D^{1,1}_\Lambda)>0$. Now if $(\|A_\lambda\|)_{\lambda\in \Lambda}\in c_0(\Lambda)$, then $D^{1,1}_\Lambda $ lies in the norm closure of $\{D^{1,1}_\Upsilon: \Upsilon\subset \Lambda\text{\ finite}\}$.     By the norm closure of $\{B:E_p\to F_r: Sz(B)\leqslant \sup_{\lambda\in \Lambda} Sz(A_\lambda)\}$ and the fact that for any finite $\Upsilon\subset \Lambda$, $Sz(D^{1,1}_\Upsilon)= \max_{\lambda\in \Upsilon} Sz(A_\lambda)$, $Sz(D^{1,1}_\Lambda)\leqslant \sup_{\lambda\in \Lambda} Sz(A_\lambda)$. By the ideal property, the reverse inequality holds. Furthermore, if $1/\gamma+1/s=1$, Theorem \ref{30r} yields that for any finite subset $\Upsilon$ of $\Lambda$, $$\alpha_{\xi, \gamma}(D^{1,1}_\Upsilon)= \|(\alpha_{\xi, \gamma}(A_\lambda))_{\lambda\in \Upsilon}\|_{\ell_s(\Upsilon)}.$$  By the norm continuity of $\alpha_{\xi, \gamma, n}$, and since $D^{1,1}_\Lambda$ lies in the norm closure of $\{D^{1,1}_\Upsilon: \Upsilon\subset \Lambda\text{\ finite}\}$, \begin{align*} \alpha_{\xi, \gamma, n}(D^{1,1}_\Lambda) & \leqslant \sup \{\alpha_{\xi, \gamma, n}(D^{1,1}_\Upsilon):\Upsilon\subset \Lambda\text{\ finite}\} \\ & \leqslant \sup \{\alpha_{\xi, \gamma}(D^{1,1}_\Upsilon):\Upsilon\subset \Lambda\text{\ finite}\}= \|(\alpha_{\xi, \gamma}(A_\lambda))_{\lambda\in \Lambda}\|_{\ell_s(\Lambda)} .\end{align*} Taking the supremum over $n\in \nn$, we deduce that $$\alpha_{\xi, \gamma}(D^{1,1}_\Lambda)\leqslant \|(\alpha_{\xi, \gamma}(A_\lambda))_{\lambda\in \Lambda}\|_{\ell_s(\Lambda)}.$$  By the ideal property, $$\alpha_{\xi, \gamma}(D^{1,1}_\Lambda) \geqslant  \sup \{\alpha_{\xi, \gamma}(D^{1,1}_\Upsilon):\Upsilon\subset \Lambda\text{\ finite}\}= \|(\alpha_{\xi, \gamma}(A_\lambda))_{\lambda\in \Lambda}\|_{\ell_s(\Lambda)},$$ giving equality.

$(ii)$ We argue similarly, except Theorem \ref{30r} gives $$\alpha_{\xi, \gamma}(D^{\infty, \infty}_\Upsilon)= \|(\alpha_{\xi, \gamma}(A_\lambda))_{\lambda\in \Lambda}\|_{\ell_\infty(\Upsilon)}$$ for any finite subset of $\Upsilon$. Here, we are using that either $(A_\lambda)_{\lambda\in \Lambda}\in c_0(\Lambda)$, or $D_\Lambda^{\infty, \infty}$ preserves an isomorph of $\ell_1$.

\end{proof}

\begin{theorem} Suppose that $p,r$ are as in one of the cases (3) or (4).  If $r=0$, suppose $1\leqslant \gamma\leqslant \infty$  and if $1\leqslant r\leqslant \infty$, suppose $1\leqslant \gamma\leqslant r$.   Then for any ordinal $\zeta$, $$\alpha_{\zeta, \gamma}(D^{p,r}_\Lambda)\leqslant \sup_{\lambda\in \Lambda} \|A_\lambda\| + \sup_{\lambda\in \Lambda} \alpha_{\zeta, \gamma}(A_\lambda).$$

\label{button classic}
\end{theorem}

\begin{proof} We prove that $$\alpha_{\zeta, \gamma, n}(D^{p,r}_\Lambda)\leqslant \sup_{\lambda\in \Lambda} \|A_\lambda\| + \sup_{\lambda\in \Lambda} \alpha_{\zeta, \gamma}(A_\lambda)$$  by induction on $n\in\nn$. Let $\beta=\sup_{\lambda\in \Lambda} \|A_\lambda\|$, $\alpha=\sup_{\lambda\in \Lambda}\alpha_{\xi, \gamma}(A_\lambda)$.    

The $n=1$ case is trivial, since $\alpha_{\zeta, \gamma, 1}(D^{p,r}_\Lambda)\leqslant \beta\leqslant \beta+\alpha$. 

Assume the result holds for $n\in \nn$ and fix $\vartheta<\alpha_{\zeta, \gamma, n+1}(D^{p,r}_\Lambda)$.    Fix $(b_i)_{i=1}^{n+1}\in B_{\ell_\gamma^{n+1}}$ and a weakly null collection $(x_t)_{t\in \Gamma_{\zeta, n+1}.D}\subset B_{E_p}$ such that $\vartheta<\inf_{t\in MAX(\Gamma_{\zeta, n+1}.D)} \|D^{p,r}_\Lambda \sum_{i=1}^{n+1} b_i z^t_i\|$.   Fix $0<\phi<1$ such that $\vartheta < \inf_{t\in MAX(\Gamma_{\zeta, n+1}.D)} \|D^{p,r}_\Lambda \sum_{i=1}^{n+1} b_i \phi z^t_i\|$.  Fix $\delta>0$. Fix any $t\in MAX(\Lambda_{\zeta, n+1, 1}.D)$ and a finite subset $\Upsilon$ of $\Lambda$ such that $\|D^{p,r}_{\Lambda\setminus \Upsilon} \sum_{s\leqslant t}\mathbb{P}_{\zeta, n+1}(s) x_s\|<\delta$. Arguing as in Lemma \ref{kimmy}, we may stabilize, relabel, and assume that there exist $(u_i, v_i)_{i=2}^{n+1}\in B_{\ell_\infty^n(\ell_p^2)}$ such that for any $2\leqslant i\leqslant n+1$ and any $t\in \Lambda_{\xi, n+1, i}.D$, $\phi\|P_\Upsilon x_t\|< u_i$ and $\phi\|P_{\Lambda\setminus \Upsilon} x_t\|< v_i$. This may be done so that the property $$\vartheta < \inf_{t\in MAX(\Gamma_{\zeta, n+1}.D)} \|D^{p,r}_\Lambda \sum_{i=1}^{n+1} b_i \phi z^t_i\|$$ is retained. Let $u_1=1$ and $v_1=0$, so that $(u_i, v_i)_{i=1}^{n+1}\in B_{\ell_\infty^{n+1}(\ell_p^2)}$.

Now using the identification of $\Gamma_{\zeta, n}.D$ with $\{s\in \Gamma_{\zeta, n+1}.D: t<s\}$ and the fact that $$\alpha_{\zeta, \gamma}(D^{p,r}_\Upsilon)= \|(\alpha_{\zeta, \gamma}(A_\lambda))_{\lambda\in \Upsilon}\|_{\ell_\infty(\Upsilon)}\leqslant \alpha$$ by Theorem \ref{30r},  we deduce using the inductive hypothesis  that each of the sets \begin{align*} \mathcal{E}_1 & := \{s\in MAX(\Gamma_{\zeta, n+1}.D): t<s, \|D^{p,r}_\Upsilon\sum_{i=1}^{n+1} b_i \phi z^s_i\| \leqslant  \delta+(\beta+\alpha) \|(b_i u_i)_{i=1}^{n+1}\|_{\ell_\gamma^{n+1}} \} \\ & \supset \{s\in MAX(\Gamma_{\zeta, n+1}.D): t<s, \|D^{p,r}_\Upsilon\sum_{i=1}^{n+1} b_i \phi z^s_i\| \leqslant  \delta+\beta u_1 + \alpha \|(b_i u_i)_{i=2}^{n+1}\|_{\ell_\gamma^n} \} \\ & \supset \{s\in MAX(\Gamma_{\zeta, n+1}.D): t<s, \|D^{p,r}_\Upsilon\sum_{i=2}^{n+1} b_i \phi z^s_i\| \leqslant  \delta + \alpha \|(b_i u_i)_{i=2}^{n+1}\|_{\ell_\gamma^n} \} \\ & \supset \{s\in MAX(\Gamma_{\zeta, n+1}.D): t<s, \|D^{p,r}_\Upsilon\sum_{i=2}^{n+1} b_iu_i \phi \frac{P_\Upsilon z^s_i}{u_i}\| \leqslant  \delta + \alpha \|(b_i u_i)_{i=2}^{n+1}\|_{\ell_\gamma^n} \}, \end{align*}

\begin{align*} \mathcal{E}_2 & :=\{s\in MAX(\Gamma_{\zeta, n+1}.D):t<s,  \|D^{p,r}_{\Lambda\setminus \Upsilon}\sum_{i=1}^{n+1} b_i \phi z^s_i\|\leqslant 2\delta+ (\beta+\alpha) \|(b_iv_i)_{i=1}^{n+1}\|_{\ell_\gamma^{n+1}}\} \\ & \supset \{s\in MAX(\Gamma_{\zeta, n+1}.D):t<s,   \|D^{p,r}_{\Lambda\setminus \Upsilon}\sum_{i=2}^{n+1} b_i \phi z^s_i\|\leqslant \delta+ (\beta+\alpha) \|(b_iv_i)_{i=1}^{n+1}\|_{\ell_\gamma^{n+1}}\} \\ & = \{s\in MAX(\Gamma_{\zeta, n+1}.D):t<s,   \|D^{p,r}_\Lambda\sum_{i=2}^{n+1} b_iv_i \frac{P_{\Lambda\setminus \Upsilon} z^s_i}{v_i}\|\leqslant \delta+ (\beta+\alpha) \|(b_iv_i)_{i=1}^{n+1}\|_{\ell_\gamma^{n+1}}\} \end{align*} is eventual. Here we are using the fact that for any $t<s\in MAX(\Gamma_{\zeta, n+1}.D)$, $z^s_1=\sum_{u\leqslant t} \mathbb{P}_{\zeta, n+1}(u)x_u$, $\|D^{p,r}_\Upsilon z^s_1\|\leqslant \beta$,  and $\|D^{p,r}_{\Lambda\setminus \Upsilon} z^s_1\|<\delta$. 

If $r=0$, let $r_1=\infty$, and otherwise let $r_1=r$. Now for any $s\in \mathcal{E}_1\cap \mathcal{E}_2$,  Minkowski's inequality and the fact that $p\leqslant r_1$  give that  \begin{align*} \vartheta & < \|D^{p,r}_\Lambda \sum_{i=1}^{n+1} b_i \phi z_i\| = \Bigl\|\Bigl(\|D^{p,r}_\Upsilon\sum_{i=1}^{n+1} b_i \phi z^s_i\|, \|D^{p,r}_{\Lambda\setminus \Upsilon}\sum_{i=1}^{n+1} b_i \phi z^s_i\| \Bigr)\Bigr\|_{\ell_r^2} \\ & \leqslant 3\delta + (\beta+\alpha)\Bigl\|\Bigl(\|(b_iu_i)_{i=1}^{n+1}\|_{\ell_\gamma^{n+1}}, \|(b_iv_i)_{i=1}^{n+1}\|_{\ell_\gamma^{n+1}}\Bigr)\Bigr\|_{\ell_{r_1}^2} \\ & \leqslant 3\delta + (\beta+\alpha)\Bigl\|\Bigl( b_i \|(u_i, v_i)\|_{\ell_{r_1}^2}\Bigr)_{i=1}^{n+1}\Bigr\|_{\ell_\gamma^{n+1}} \\ & \leqslant 3\delta + (\beta+\alpha)\Bigl\|\Bigl( b_i \|(u_i, v_i)\|_{\ell_p^2}\Bigr)_{i=1}^{n+1}\Bigr\|_{\ell_\gamma^{n+1}} \leqslant 3\delta+ \beta+\alpha.\end{align*} Since $\delta>0$ and $\vartheta<\alpha_{\zeta, \gamma,n+1}(D^{p,r}_\Lambda)$ were arbitrary, we are done.

\end{proof}

\begin{corollary} In either of cases (3), (4), $$Sz(D^{p,r}_\Lambda)\leqslant \bigl(\sup_{\lambda\in \Lambda} Sz(A_\lambda)\bigr)\omega.$$

\end{corollary}

\begin{proof} If $r=0$, let $\gamma=\infty$, and otherwise let $\gamma=r$. Let $\sup_{\lambda\in \Lambda} Sz(A_\lambda)=\omega^\xi$. Then for every $\lambda\in \Lambda$, $\alpha_{\xi, \gamma, 1}(A_\lambda)=\alpha_{\xi, \gamma}(A_\lambda)=0$, whence $$\alpha_{\xi, \gamma}(D^{p,r}_\Lambda) \leqslant \sup_{\lambda\in \Lambda} \|A_\lambda\|<\infty$$ by Theorem \ref{button classic}.   By Remark \ref{snit}, $Sz(D^{p,r}_\Lambda)\leqslant \omega^{\xi+1}=\omega^\xi \omega$.

\end{proof}

\begin{corollary} Suppose $p,r$ are as in either case (3) or case (4).  For any ordinal $\zeta$, $$\underset{m\in \nn}{\inf} \underset{\lambda\in \Lambda}{\ \sup\ }  \theta_{\zeta, m}(A_\lambda)= \underset{m\in \nn}{\inf} \theta_{\zeta, m}(D^{p,r}_\Lambda)= \theta_{\zeta+1, 1}(D^{p,r}_\Lambda).$$   

In particular, for an ordinal $\xi>0$, $$\inf \{\theta_{\xi,1}(D^{p,r}_{\Lambda\setminus \Upsilon}): \Upsilon\subset \Lambda\text{\ finite}\}=0$$ if and only if for every $\ee>0$, there exist a finite subset $\Upsilon$ of $\Lambda$, an ordinal $\zeta<\xi$, and $m\in \nn$ such that $\sup_{\lambda\in \Upsilon} \theta_{\zeta, m}(A_\lambda)\leqslant \ee$ if and only if for every $\ee>0$, there exists a finite subset $\Upsilon$ of $\Lambda$ such that $\sup_{\lambda\in \Lambda\setminus \Upsilon} Sz(A_\lambda, \ee)<\omega^\xi$. 
\label{fir}
\end{corollary}

\begin{proof} If $r=0$, let $\gamma=\infty$, and otherwise let $\gamma=r$. Let $1/\gamma+1/s=1$. Note that since $r\neq 1$, $\gamma\neq 1$, so $s<\infty$.    It follows from Proposition \ref{too easy} that $\underset{m\in \nn}{\inf} \theta_{\zeta, m}(D^{p,r}_\Lambda)= \theta_{\zeta+1, 1}(D^{p,r}_\Lambda)$, and it follows from the fact that $\theta_{\zeta, m}$ is an ideal seminorm that $\underset{m\in \nn}{\inf} \underset{\lambda\in \Lambda}{\ \sup\ }  \theta_{\zeta, m}(A_\lambda)\leqslant  \underset{m\in \nn}{\inf} \theta_{\zeta, m}(D^{p,r}_\Lambda)$. Fix $\ee_1>\inf_m  \underset{\lambda\in \Lambda}{\sup} \theta_{\zeta, m}(A_\lambda)$. Let $m\in \nn$ and $\ee_2, \ee$ be such that $\ee_1>\ee_2>\ee>\theta_{\zeta, m}(A_\lambda)$ for all $\lambda\in \Lambda$.  Let $R=\sup_{\lambda\in \Lambda}\|A_\lambda\|$ and fix $n\in \nn$, $\delta>0$ such that $Rm^{1/s}/n^{1/s}<\ee_2-\ee$,  $ \delta+R/n^{1/s}<\ee_1-\ee_2$, and $\delta+\ee+Rm^{1/s}/n^{1/s}<\ee_2$.   

Fix a weakly null collection $(x_t)_{t\in \Gamma_{\zeta, n^2}.D}\subset B_{E_p}$. Let $\Upsilon_0=\varnothing$, $t_0=\varnothing$, and fix any $t_1\in MAX(\Lambda_{\zeta, n^2, n}.D)$.  Now suppose we have chosen $t_1<\ldots <t_{k-1}$ and finite sets $\Upsilon_1 \subset \ldots \subset \Upsilon_{k-1}\subset \Lambda$ such that for each $1\leqslant i<k$, $t_i\in MAX(\Lambda_{\zeta, n^2, in}.D)$, $$\|D^{p,r}_{\Lambda\setminus \Upsilon_i}\sum_{j=(i-1)n+1}^{in} \sum_{t_i\geqslant s\in \Lambda_{\zeta, n^2, j}.D} \mathbb{P}_{\zeta, n^2}(s) x_s\|<\delta$$ and $$\|D^{p,r}_{\Upsilon_{i-1}} n^{-1}\sum_{j=(i-1)n+1}^{in} \sum_{t_i\geqslant s\in \Lambda_{\zeta, n^2, j}.D} \mathbb{P}_{\zeta, n^2}(s) x_s\|<\ee_2.$$  Now since we may identify $$\{s\in \cup_{j=(k-1)n+1}^{kn} \Lambda_{\zeta, n^2, j}.D: t_{k-1}<s\}$$ with $\Gamma_{\zeta, n}.D$, by Theorem \ref{30r}, there exists $t_k\in MAX(\Lambda_{\zeta, n^2, kn}.D)$ such that \begin{align*} \|D^{p,r}_{\Upsilon_{k-1}} \sum_{j=(k-1)n+1}^{kn} \sum_{t_k\geqslant s\in \Lambda_{\zeta, n^2, j}.D}n^{-1}\mathbb{P}_{\zeta, n^2}(s) x_s\| & \leqslant  \delta + \ee\|(1/n)_{j=1}^n\|_{\ell_1^n} + Rm^{1/s}\|(1/n)_{j=1}^n\|_{\ell_\gamma^n} \\ & = \delta+\ee +R\frac{m^{1/s}}{n^{1/s}}<\ee_2.\end{align*}  Here we are using that $\ell_r^n=\ell_\gamma^n$.    Fix $\Upsilon_k\subset \Lambda$ finite such that $\Upsilon_{k-1}\subset \Upsilon_k$ and $$\|D^{p,r}_{\Lambda\setminus \Upsilon_k} \sum_{j=(k-1)n+1}^{kn} \sum_{t_k\geqslant s\in \Lambda_{\zeta, n^2, j}}n^{-1}\mathbb{P}_{\zeta, n^2}(s) x_s\|<\delta.$$  These completes the recursive construction. Let $t=t_n$ and for each $1\leqslant i\leqslant n$, let $F_i=\Upsilon_i\setminus \Upsilon_{i-1}$.  For each $1\leqslant i\leqslant n$, let $z_i=n^{-1}\sum_{t_{i-1}<s\leqslant t_i} \mathbb{P}_{\zeta, n^2}(s)x_s$.   Then \begin{align*}\|D^{p,r}_\Lambda \sum_{j=1}^{n^2} n^{-2}z^t_i\| & = \|D^{p,r}_\Lambda n^{-1}\sum_{i=1}^n z_i\| \\ & = \|D^{p,r}_\Lambda n^{-1}\sum_{i=1}^n [P_{\Upsilon_{i-1}}+P_{F_i}+P_{\Lambda\setminus \Upsilon_i}]z_i\|  \\ & \leqslant \ee_2 + \|n^{-1}\sum_{i=1}^n P_{F_i} z_i\|+ \delta \\ & \leqslant \ee_2+ R/n^{1/s}+\delta<\ee_1.  \end{align*}  From this it follows that $$\inf \{\theta_{\zeta, k}(D^{p,r}_\Lambda): k\in \nn\} \leqslant \theta_{\zeta, n^2}(D^{p,r}_\Lambda)<\ee_1.$$  Since $\ee_1>\underset{k\in\nn}{\ \inf\ } \underset{\lambda\in \Lambda}{\ \sup\ } \theta_{\zeta, k}(A_\lambda)$ was arbitrary, $$\underset{m\in \nn}{\inf} \underset{\lambda\in \Lambda}{\ \sup\ }  \theta_{\zeta, m}(A_\lambda)\geqslant  \underset{m\in \nn}{\inf} \theta_{\zeta, m}(D^{p,r}_\Lambda).$$

The last statement follows immediately from the previous statements.

\end{proof}

We now eliminate a trivial case. The proof is obvious, so we omit it.

\begin{proposition} Suppose that $\sup_{\lambda\in \Lambda}Sz(A_\lambda)=1$.  Assume that $p,r$ are as in one of cases (3),(4) above. Then $D^{p,r}_\Lambda$ is compact, $\theta_{0,1}(D^{p,r}_\Lambda)=\alpha_{0, \gamma}(D^{p,r}_\Lambda)=0$ for all $1\leqslant \gamma\leqslant \infty$, and  $Sz(D^{p,r}_\Lambda)=1$ if and only  if $(\|A_\lambda\|)_{\lambda\in \Lambda}\in c_0(\Lambda)$. If $(\|A_\lambda\|)_{\lambda\in \Lambda}\in \ell_\infty(\Lambda)\setminus c_0(\Lambda)$, then $Sz(D^{p,r}_\Lambda)=\omega$, $$\theta_{0,1}(D^{p,r}_\Lambda)=\inf\{\|D^{p,r}_{\Lambda\setminus \Upsilon}\|: \Upsilon\subset \Lambda\text{\ finite}\}>0$$ and \begin{enumerate}[(i)]\item if $r=0$ or $r=\infty$, $\alpha_{0,\gamma}(D^{p,r}_\Lambda)=\theta_{0,1}(D^{p,r}_\Lambda)<\infty$ for any $1\leqslant \gamma\leqslant \infty$, \item if $1\leqslant p\leqslant r<\infty$ and $1<r$, $\alpha_{0, \gamma}(D^{p,r}_\Lambda)=\theta_{0,1}(D^{p,r}_\Lambda)$ for all $1\leqslant \gamma\leqslant r$ and $\alpha_{0,\gamma}(D^{p,r}_\Lambda)=\infty$ for all $r<\gamma\leqslant \infty$. \end{enumerate}

\end{proposition}

The next theorem incorporates Brooker's result about the Szlenk index, but also includes new information regarding the behavior of the $\theta_{\zeta, m}(A_\lambda)$ quantities, which are only indirectly related to $Sz(A_\lambda, \ee)$. 

\begin{theorem} Suppose $p,r$ are as in either case (3) or case (4) and  $\sup_{\lambda\in \Lambda}Sz(A_\lambda)=\omega^\xi>1$.  The following are equivalent. \begin{enumerate}[(i)]\item $Sz(D^{p,r}_\Lambda)=\omega^\xi$. \item For every $\ee>0$, $\sup_{\lambda\in \Lambda} Sz(A_\lambda, \ee)<\omega^\xi$. \item $\inf\{\sup_{\lambda\in \Lambda} \theta_{\zeta, m}(A_\lambda): \zeta<\xi, m\in \nn\}=0$. . \end{enumerate} 

\label{db}

\end{theorem}

\begin{proof}$(i)\Rightarrow (iii)$ If $Sz(D^{p,r}_\Lambda)=\omega^\xi$, then by Proposition \ref{too easy} and the fact that $\theta_{\zeta, m}$ is an ideal seminorm, $$0= \theta_{\xi, 1}(D^{p,r}_\Lambda)= \inf \{\theta_{\zeta, m}(D^{p,r}_\Lambda): \zeta<\xi, m\in \nn\} \geqslant \inf\{\sup_{\lambda\in \Lambda} \theta_{\zeta, m}(A_\lambda): \zeta<\xi, m\in \nn\}\geqslant 0.$$

$(iii)\Rightarrow (i)$ We first remark that $Sz(D^{p,r}_\Lambda)\geqslant \omega^\xi$. Indeed, for any $\zeta<\xi$, there exists $\lambda\in \Lambda$ such that $Sz(A_\lambda)>\omega^\zeta$, so $\theta_{\zeta, 1}(D^{p,r}_\Lambda)\geqslant \theta_{\zeta, 1}(A_\lambda)>0$, and $Sz(D^{p,r}_\Lambda)>\omega^\zeta$. Since $Sz(D^{p,r}_\Lambda)$ must be of the form $\omega^\gamma$ for some $\gamma$, it follows that $Sz(D^{p,r}_\Lambda)\geqslant \omega^\xi$.

If $\xi$ is a successor, say $\xi=\eta+1$, then by Proposition \ref{too easy} and Corollary \ref{fir},  \begin{align*} 0=\inf\{\sup_{\lambda\in \Lambda} \theta_{\zeta, m}(A_\lambda): \zeta<\xi, m\in \nn\} & = \inf\{\sup_{\lambda\in \Lambda}\theta_{\eta, m}(A_\lambda): m\in \nn\} = \theta_{\eta+1, 1}(D^{p,r}_\Lambda)=\theta_{\xi, 1}(D^{p,r}_\Lambda), \end{align*} so $Sz(D^{p,r}_\lambda)\leqslant \omega^\xi$.

Now if $\xi$ is a limit, then by Proposition \ref{too easy}, for any $\ee>0$, there exist $\zeta<\xi, m\in \nn$ such that $\sup_{\lambda\in \Lambda} \theta_{\zeta, m}(A_\lambda)<\ee$.  Then by Corollary \ref{fir}, and Proposition \ref{too easy}, $\theta_{\xi, 1}(D^{p,r}_\Lambda)\leqslant \theta_{\zeta+1, 1}(D^{p,r}_\Lambda) \leqslant \ee$.   Since this holds for any $\ee>0$, $\theta_{\xi,1}(D^{p,r}_\Lambda)=0$.

$(ii)\Rightarrow (iii)$ Fix $\ee>0$ and $\zeta<\xi$, $m\in \nn$ such that $Sz(A_\lambda, \ee/3)\leqslant \omega^\zeta m$ for all $\lambda\in \Lambda$.  Fix $n\in \nn$ such that $Rm/n<2\ee/3$, where $R=\sup_{\lambda\in \Lambda}\|A_\lambda\|$. Now if $\lambda\in \Lambda$ is such that $\theta_{\zeta, n}(A_\lambda)>\ee$, there exist a directed set $D$, a weakly null collection $(x_t)_{t\in \Gamma_{\zeta, n}.D}\subset B_{E_p}$, functionals $(y^*_t)_{t\in MAX(\Gamma_{\zeta, n}.D)}\subset B_{Y^*_\lambda}$, and real numbers $b_1, \ldots, b_n$ such that $\sum_{i=1}^n \frac{b_i}{n}>\ee$ and for each $1\leqslant i\leqslant n$ and $\Lambda_{\zeta, n, i}.D\ni s\leqslant t\in MAX(\Gamma_{\zeta, n}.D)$, $\text{Re\ }y^*_t(x_s)\geqslant b_i$.   Let $T=\{i\leqslant n: b_i>2\ee/3\}$ and fix $1/2<\psi<1$.  We claim that $|T|<m$. If it were not so, we could find a level map $d:\Gamma_{\zeta, m}.D\to \Gamma_{\zeta, n}.D$ such that $(x_{d(t)})_{t\in \Gamma_{\zeta, m}.D}$ is weakly null and  $d(\Lambda_{\zeta, m, i})\subset \Lambda_{\zeta, n, l_i}.D$, where $l_1<\ldots<l_m$, $l_i\in T$. Then if $e$ is any extension of $d$, the collections $(x_{d(t)})_{t\in \Gamma_{\zeta, m}.D}$, $(y^*_{e(t)})_{t\in MAX(\Gamma_{\zeta, m}.D)}$ witness that $$\varnothing\neq s^{\omega^\zeta}_{\psi\ee_{l_m}}\ldots s^{\omega^\zeta}_{\psi\ee_{l_1}}(A^*_\lambda B_{Y^*_\lambda}) \supset s^{\omega^\zeta}_{\psi 2\ee/3}(A^*_\lambda B_{Y^*_\lambda})\supset s^{\omega^\zeta m}_{\ee/3}(A^*_\lambda B_{Y^*_\lambda}),$$ a contradiction. Thus $\sup_{\lambda\in \Lambda} \theta_{\zeta, n}(A_\lambda)\leqslant \ee$, whence $$\inf \{\sup_{\lambda\in \Lambda} \theta_{\zeta, m}(A_\lambda): \zeta<\xi, m\in \nn\}=0.$$

$(iii)\Rightarrow (ii)$ We recall that  if for some $\lambda\in \Lambda$, $\ee>0$, $\zeta\in \textbf{Ord}$, and $m\in \nn$,  $Sz(A_\lambda, \ee)>\omega^\zeta m$, then $\theta_{\zeta, m}(A_\lambda)\geqslant\ee/4$. From this it follows that since $\inf\{\sup_{\lambda\in \Lambda}(A_\lambda): \zeta<\xi, m\in \nn\}<\ee/4$, there exist $\zeta<\xi$ and $m\in\nn$ such that $Sz(A_\lambda, \ee)\leqslant \omega^\zeta m<\omega^\xi$ for all $\lambda\in \Lambda$.

\end{proof}

We next elucidate the behavior of $\alpha_{\xi, \gamma}(D^{p,r}_\Lambda)$ for cases (3) and (4).

\begin{theorem} Suppose that $p,r$ are as in either case (3) or case (4). Suppose that $Sz(D^{p,r}_\Lambda)=\omega^{\xi+1}$ and $\sup_{\lambda\in \Lambda} Sz(A_\lambda)=\omega^\zeta$.

If $\zeta=\xi$, $\alpha_{\xi, \gamma}(D^{p,r}_\Lambda)<\infty$ if and only if either \begin{enumerate}[(i)]\item $r=0$,  or \item $1\leqslant \gamma\leqslant r$. \end{enumerate}

If $\zeta=\xi+1$, $\alpha_{\xi, \gamma}(D^{p,r}_\Lambda)<\infty$ if and only if either  \begin{enumerate}[(i)]\item $r=0$ and $(\alpha_{\xi, \gamma}(A_\lambda))_{\lambda\in \Lambda}\in \ell_\infty(\Lambda)$, \item $1\leqslant \gamma\leqslant r\leqslant \infty$ and $(\alpha_{\xi, \gamma}(A_\lambda))_{\lambda\in \Lambda}\in \ell_\infty(\Lambda)$, or \item $1\leqslant r<\gamma\leqslant \infty$, $\inf\{\theta_{\xi, 1}(D^{p,r}_{\Lambda\setminus \Upsilon}: \Upsilon\subset\Lambda\text{\ \emph{finite}}\}=0,$ and $(\alpha_{\xi, \gamma}(A_\lambda))_{\lambda\in \Lambda}\in \ell_s(\Lambda)$, where $1/\gamma+1/s=1/r$. \end{enumerate}

\label{zg}

\end{theorem}

\begin{proof} First suppose that $Sz(D^{p,r}_\Lambda)=\omega^{\xi+1}$ and $\sup_{\lambda\in \Lambda} Sz(A_\lambda)=\omega^\xi$.  By Theorem \ref{button classic},  if $r=0$ or $1\leqslant p\leqslant r\leqslant \infty$ and $1\leqslant \gamma\leqslant r$, $$\alpha_{\xi, \gamma} (D^{p,r}_\Lambda)\leqslant \sup_{\lambda\in \Lambda} \|A_\lambda\|+ \sup_{\lambda\in \Lambda}\alpha_{\xi, \gamma}(A_\lambda)= \sup_{\lambda\in \Lambda}\|A_\lambda\|<\infty,$$ since $\alpha_{\xi, \gamma}(A_\lambda)=0$ for each $\lambda\in \Lambda$. Now suppose that $1\leqslant p\leqslant r<\gamma$ and let $1/\gamma+1/s=1/r$. By Theorem \ref{db}, since $Sz(D^{p,r}_\Lambda)>\sup_{\lambda\in \Lambda} Sz(A_\lambda)$, this means $$\inf \{\sup_{\lambda\in \Lambda}\theta_{\eta, m}(A_\lambda): \eta<\xi, m\in \nn\}>\ee>0.$$    Then for any $n\in \nn$, we may partition $\Lambda$ into sets $\Lambda_1, \ldots, \Lambda_n$ such that for each $1\leqslant i\leqslant n$, $\inf \{\sup_{\lambda\in \Lambda_i}\theta_{\eta, m}(A_\lambda): \eta<\xi, m\in\nn\}>\ee$, whence $\theta_{\xi,1}(D^{p,r}_{\Lambda_i})>\ee$ for each $1\leqslant i\leqslant n$. Then we may find $(x_t)_{t\in \Gamma_{\xi, n}.D}\subset B_{E_p}$ weakly null such that for any $1\leqslant i\leqslant n$, $\inf_{t\in MAX(\Gamma_{\xi,n}.D)} \|D^{p,r}_{\Lambda_i}\sum_{t\geqslant s\in \Lambda_{\xi, n,i}.D} \mathbb{P}_{\xi, n}(s) x_s\|>\ee$ and for any $t\in \Lambda_{\xi,n,i}.D$, $x_t=P_{\Lambda_i} x_t$.   Then $$\alpha_{\xi, \gamma, n}(D^{p,r}_\Lambda)\geqslant \inf_{t\in MAX(\Gamma_{\xi,n}.D)} \|D^{p,r}_\Lambda\sum_{i=1}^n z^t_i\| \geqslant \ee n^{1/s}.$$  This yields that $\alpha_{\xi, \gamma}(D^{p,r}_\Lambda)=\infty$ if $1\leqslant p\leqslant r<\gamma$.

Now suppose $Sz(D^{p,r}_\Lambda)=\sup_{\lambda\in \Lambda} Sz(A_\lambda)=\omega^{\xi+1}$.   Then if either $r=0$ or $1\leqslant p\leqslant r\leqslant \infty$ and $1\leqslant \gamma\leqslant r$,  by Theorem \ref{button classic}, $$\sup_{\lambda\in \Lambda} \alpha_{\xi, \gamma}(A_\lambda)\leqslant \alpha_{\xi, \gamma}(D^{p,r}_\Lambda)\leqslant \sup_{\lambda\in \Lambda} \|A_\lambda\|+\sup_{\lambda\in \Lambda} \alpha_{\xi, \gamma}(A_\lambda),$$ so that $\alpha_{\xi, \gamma}(D^{p,r}_\Lambda)<\infty$ if and only if $\sup_{\lambda\in \Lambda}\alpha_{\xi, \gamma}(A_\lambda)<\infty$.

We now assume that $1\leqslant p\leqslant r<\gamma\leqslant \infty$ and $1/\gamma+1/s=1/r$.  We will show that $\alpha_{\xi, \gamma}(D^{p,r}_\Lambda)<\infty$ if and only if $(\alpha_{\xi, \gamma}(A_\lambda))_{\lambda\in \Lambda}\in \ell_s(\Lambda)$ and $\inf\{\theta_{\xi, 1}(D^{p,r}_{\Lambda\setminus \Upsilon}): \Upsilon\subset \Lambda\text{\ finite}\}=0$.   First suppose that $(\alpha_{\xi, \gamma}(A_\lambda))_{\lambda\in \Lambda}\in \ell_s(\Lambda)$ and $\inf\{\theta_{\xi, 1}(D^{p,r}_{\Lambda\setminus \Upsilon}: \Upsilon\subset \Lambda\text{\ finite}\}=0$.  Then for any finite subset $\Upsilon$ of $\Lambda$ and $n\in \nn$,  by Theorem \ref{30r}, \begin{align*}  \alpha_{\xi, \gamma}(D^{p,r}_\Lambda) & \leqslant \alpha_{\xi, \gamma,n}(D^{p,r}_\Upsilon) + \alpha_{\xi, \gamma,n}(D^{p,r}_{\Lambda\setminus \Upsilon}) \\ & \leqslant \|(\alpha_{\xi, \gamma}(A_\lambda))_{\lambda\in \Upsilon}\|_{\ell_s(\Upsilon)} + n \alpha_{\xi, \gamma, 1}(D^{p,r}_{\Lambda\setminus \Upsilon}) \\ & \leqslant \|(\alpha_{\xi, \gamma}(A_\lambda))_{\lambda\in \Lambda}\|_{\ell_s(\Lambda)} + n \theta_{\xi, \gamma, 1}(D^{p,r}_{\Lambda\setminus \Upsilon}). \end{align*} Taking the infimum over $\Upsilon\subset \Lambda$ and then the supremum over $n\in \nn$ yields that $\alpha_{\xi, \gamma,}(D^{p,r}_\Lambda)\leqslant \|(\alpha_{\xi, \gamma}(A_\lambda))_{\lambda\in \Lambda}\|_{\ell_s(\Lambda)}<\infty$.

Now assume that one of the conditions $(\alpha_{\xi, \gamma}(A_\lambda))_{\lambda\in \Lambda}\in \ell_s(\Lambda)$ and $\inf\{\theta_{\xi, 1}(D^{p,r}_{\Lambda\setminus \Upsilon}): \Upsilon\subset \Lambda\text{\ finite}\}=0$ fails. If $(\alpha_{\xi, \gamma}(A_\lambda))_{\lambda\in \Lambda}\notin \ell_s(\Lambda)$, then by Theorem \ref{30r}, $$\alpha_{\xi, \gamma}(D^{p,r}_\Lambda)   \geqslant \sup\{\alpha_{\xi, \gamma}(D^{p,r}_\Upsilon): \Upsilon\subset \Lambda\text{\ finite}\} = \|(\alpha_{\xi, \gamma}(A_\lambda))_{\lambda\in \Lambda}\|_{\ell_s(\Lambda)}=\infty.$$

Now assume that $\inf\{\theta_{\xi, 1}(D^{p,r}_{\Lambda\setminus \Upsilon}): \Upsilon\subset \Lambda\text{\ finite}\}>2\ee>0$. Fix $n\in \nn$, a weak neighborhood basis $D$ at $0$ in $E_p$, and $\delta>0$ such that $n\delta<\ee n^{1/r}$, where $1/r+1/r'=1$.  Let $\Upsilon_\varnothing=\varnothing$.   Since $\theta_{\xi, 1}(D^{p,r}_\Lambda)>2\ee$, we may, by identifying $\Lambda_{\xi, n,1}.D$ with $\Gamma_{\xi, 1}.D$, fix a weakly null collection $(x_t)_{t\in \Lambda_{\xi, n,1}.D}\subset B_{E_p}$ such that $\inf_{t\in MAX(\Lambda_{\xi, n, 1}.D)} \|D^{p,r}_\Lambda \sum_{s\leqslant t} \mathbb{P}_{\xi,n}(s)x_s\|>2\ee$. For each $t\in MAX(\Lambda_{\xi, n,1}.D)$, fix a finite set $\Upsilon_t$ such that $\|D^{p,r}_{\Lambda\setminus \Upsilon_t} \sum_{s\leqslant t} \mathbb{P}_{\xi,n}(s)x_s\|<\ee$.         Now assume that for some $i<n$, a weakly null collection $(x_t)_{t\in \cup_{j=1}^i \Lambda_{\xi, n, j}.D}$ has been chosen. Suppose also that for each $1\leqslant j\leqslant i$ and each $t\in MAX(\Lambda_{\xi, n,j}.D)$, a finite subset $\Upsilon_t$ has been chosen such that if $t>s\in MAX(\Lambda_{\xi, n, k}.D)$ for some $1\leqslant k<j$, $\Upsilon_s\subset \Upsilon_t$, and $\|D^{p,r}_{\Lambda\setminus \Upsilon_t} \sum_{\Lambda_{\xi, n,j}.D\ni u\leqslant t} \mathbb{P}_{\xi,n}(u)x_u\|<\ee$.    For each $t\in MAX(\Lambda_{\xi, n,i}.D)$, since $\theta_{\xi,1}(D^{p,r}_{\Lambda\setminus \Upsilon_t})>2\ee$, by identifying $\Gamma_{\xi,1}.D$ with $\{s\in \Lambda_{\xi, n,i+1}.D: t<s\}$, we may fix a collection $(x_s)_{t<s\in \Lambda_{\xi,n,i+1}.D}\subset B_{E_p}$ such that $P_{\Lambda\setminus \Upsilon_t}x_s=x_s$ and $$\inf \{\|D^{p,r}_\Lambda\sum_{t<u\leqslant s}\mathbb{P}_{\xi,n}(u)x_u\|: t<s\in MAX(\Lambda_{\xi, n, i+1})\}>2\ee.$$  Now for each $s\in MAX(\Lambda_{\xi, n, i+1})$ such that $t<s$, fix a finite set $\Upsilon_s\subset \Lambda$ such that $\Upsilon_t\subset \Upsilon_s$ and $$\|D^{p,r}_{\Lambda\setminus \Upsilon_s}\sum_{t<u\leqslant s} \mathbb{P}_{\xi, n}(u)x_u\|<\ee.$$    This completes the recursive choice of a weakly null collection $(x_t)_{t\in \Gamma_{\xi, n}.D}\subset B_{E_p}$.  Now for any $t\in MAX(\Gamma_{\xi, n}.D)$, let $\varnothing=t_0<t_1<\ldots <t_n$ be such that $t_i\in MAX(\Lambda_{\xi, n, i}.D)$ for each $1\leqslant i\leqslant n$ and let $\varnothing\subset \Upsilon_{t_1}\subset \ldots <\Upsilon_{t_n}$ be as in the choice of $(x_s)_{s\in \Gamma_{\xi, n}.D}$.    Now let $F_i=\Upsilon_{t_i}\setminus \Upsilon_{t_{i-1}}$, $z_i=\sum_{t\geqslant s\in \Lambda_{\xi, n,i}.D}\mathbb{P}_{\xi,n}(s)x_s$,  and note that \begin{align*}\|D^{p,r}_\Lambda \sum_{i=1}z_i \| & \geqslant \|D^{p,r}_\Lambda \sum_{i=1}^n P_{F_i}z_i\|-\sum_{i=1}^n\|D^{p,r}_{\Lambda\setminus \Upsilon_{t_i}}z_i\| \geqslant 2\ee n^{1/r} - n\delta>\ee n^{1/r}. \end{align*} From this it follows that $$\alpha_{\xi, \gamma}(D^{p,r}_\Lambda) \geqslant \sup_n \ee n^{\frac{1}{r}-\frac{1}{\gamma}}=\infty. $$

\end{proof}

\end{document}